\documentclass[11pt]{amsart}

\usepackage{amssymb,amsmath,accents}
\usepackage{bbm}
\usepackage{graphicx}
\usepackage{a4wide}

\newtheorem{theorem}{Theorem}
\newtheorem{prop}{Proposition}
\newtheorem{lemma}{Lemma}    
\newtheorem{fact}{Fact} 
\newtheorem{coro}{Corollary}
\theoremstyle{definition}

\newcommand{\ts}{\hspace{0.5pt}}
\newcommand{\nts}{\hspace{-0.5pt}}
\newcommand{\AAA}{\mathbb{A}}

\newcommand{\RR}{\mathbb{R}\ts}

\newcommand{\II}{\mathbb{I}}

\newcommand{\MM}{\mathbb{M}}

\newcommand{\PP}{\mathbb{P}}
\newcommand{\VV}{\mathbb{V}}

\newcommand{\cA}{\mathcal{A}}
\newcommand{\cB}{\mathcal{B}}
\newcommand{\cC}{\mathcal{C}}
\newcommand{\cD}{\mathcal{D}}
\newcommand{\cE}{\mathcal{E}}
\newcommand{\cF}{\mathcal{F}}
\newcommand{\cG}{\mathcal{G}}
\newcommand{\cH}{\mathcal{H}}
\newcommand{\cM}{\mathcal{M}}
\newcommand{\cP}{\mathcal{P}}

\newcommand{\pa}{\hphantom{g}\nts\nts}

\newcommand{\bs}{\boldsymbol}
\newcommand{\dd}{\,\mathrm{d}}
\newcommand{\chn}{\ts\mathrm{chn}}
\newcommand{\supp}{\ts\mathrm{supp}}
\newcommand{\ee}{\ts\mathrm{e}}
\newcommand{\one}{\mathbbm{1}}
\newcommand{\rtot}[1]{\varrho^{#1}_{\mathrm{tot}}}

\newcommand{\pmin}{\ts\ts\underline{\nts\nts 0\nts\nts}\ts\ts}
\newcommand{\pmax}{\ts\ts\underline{\nts\nts 1\nts\nts}\ts\ts}
\newcommand{\bigtim}{\mbox{\LARGE $\times$}}
\newcommand{\udo}[1]{\underaccent{$\text{.}$}{#1\ts}\nts}

\begin{document}

\title
{The recombination equation 
for interval partitions}

\author{Michael Baake}
\address{Fakult\"at f\"ur Mathematik, Universit\"at Bielefeld, 
         Postfach 100131, 33501 Bielefeld, Germany}

\author{Elham Shamsara}
\address{Department of Applied Mathematics, School of
 Mathematical Sciences, \newline
 \indent Ferdowsi University of Mashhad (FUM), Mashhad, Iran}



\begin{abstract} 
  The general deterministic recombination equation in continuous time
  is analysed for various lattices, with special emphasis on the
  lattice of interval (or ordered) partitions. Based on the recently
  constructed \cite{main} general solution for the lattice of all
  partitions, the corresponding solution for interval partitions is
  derived and analysed in detail. We focus our attention on the
  recursive structure of the solution and its decay rates, and also
  discuss the solution in the degenerate cases, where it comprises
  products of monomials with exponentially decaying factors. This can
  be understood via the Markov generator of the underlying
  partitioning process that was recently identified.  We use interval
  partitions to gain insight into the structure of the solution, while
  our general framework works for arbitrary lattices.
\end{abstract}

\maketitle

\section{Introduction}

Recombination is an important mechanism in population genetics. It is
one of the fundamental processes that underlies the time evolution of
a population of individuals under sexual reproduction; see
\cite{Buerger,L,Durrett} for general background and
\cite{Gei,Ben,D} for important contributions. Recombination is a
stochastic process that, in continuous time and in the deterministic
limit of large populations, leads to a nonlinear ordinary differential
equation (ODE) in the Banach space of finite measures on a locally
compact product space for the types that characterise the individuals;
see \cite{BB,main} and references therein. The flow of this ODE
generally acts on an infinite-dimensional space, but has a number of
nice properties, including preservation of positivity under the
forward flow as well as norm preservation in the positive sector.

Moreover, the recombination ODE admits a reduction to a
finite-dimensional nonlinear system of ODEs via a suitable ansatz in
form of a finite convex combination of probability measures that
derive from the initial condition by the application of a finite set
of (also nonlinear) operators, which are known as
\emph{recombinators}. The explicit solution of this ODE system was
recently achieved in \cite{main}, building on previous work
\cite{BB,MB,EB-ICM,P}, in a recursive way. The solution formula
derived there was explicit enough to establish the exponentially fast
convergence of the solution to a unique equilibrium that only depends
on the initial condition, as expected from the previously known cases
\cite{BB} as well as the general theory \cite{Buerger}; see also the
indroduction of \cite{main} and references therein for more.

The mathematical setting was based on the lattice $\PP (S)$ of all
partitions of a finite set, say $S = \{ 1,2, \dots, n \}$, and was not
restricted to the biologically motivated case of two parents, but also
allowed for an arbitrary recombination scheme.  This results in an
interesting nonlinear ODE system in its own right, and adds a
nontrivial example of a solvable case. Nevertheless, the treatment in
\cite{main} did not consider other lattices and thus did not explore
the full structure of the solution in lattice-theoretic terms. It is
the aim of the present paper to continue this development, in
particular by employing the lattice of ordered or \emph{interval
  partitions}.  This has the advantage that complete closed
expressions can be derived for sizes up to four sites, and with some
effort even for five sites. This way, the intricate and rather complex
nature of the recursive approach becomes more transparent, and a
better understanding of the solution is possible. Also, we are able to
shed more light on the non-generic degenerate cases with singularities
in the recursion, which were only briefly looked at in \cite{main} and
had been excluded from consideration in all previous attempts
\cite{L,D}. Finally, motivated by the underlying partitioning process
from \cite[Sec.~6]{main} and its Kolmogorov backward and forward
equations, we derive a \emph{linear} ODE system for the solution of
the recombination ODE, thus explaining why the nonlinear ODE was
solvable in the first place and how the degeneracies emerge.
\smallskip

The paper is organised as follows. In Section~\ref{sec:prelim}, we
introduce the mathematical concepts from lattice theory that we need,
with some focus on our later applications (with one result of general
interest being added as an Appendix).  This is followed by a review
of some of the known results for the recombination ODE in
Section~\ref{sec:recall}, adapted to our more general setting.  This
section also includes some new results on the decay rates of the
system.  Then, we consider the ODE for the lattice $\II (S)$ of
interval partitions in some detail in Section~\ref{sec:int-part}. This
lattice allows to explore the general solution without the
combinatorial complexity of the full lattice $\PP (S)$, and leads to a
number of previously unknown properties. Special cases (with up to
five sites) are written out in explicit form (and with possible
applications in mind) in Sections~\ref{sec:examples} and
\ref{sec:five-sites}, before we sketch further extensions and
directions in Section~\ref{sec:outlook}, including the derivation
of a linear ODE system for the recombination process.

\section{Preliminaries}
\label{sec:prelim}

Let $S$ be a finite set, for instance $S = \{ 1,2,\dots , n\}$, and
let $\PP (S)$ denote the lattice of all partitions of $S$, with order
relation $\preccurlyeq$, so that $\cA \preccurlyeq \cB$ stands for
$\cA$ being a refinement of $\cB$.  Below, we use the shorthands
$\pmax = \{S \}$ and $\pmin = \big\{ \{1\}, \{2\}, \dots ,\{n\}\big\}$
for the maximal and the minimal element of $\PP (S)$. We use
$\cA\land\cB$ for the coarsest refinement and $\cA\lor\cB$ for the
finest coarsening of $\cA$ and $\cB$.
 
Let $\VV \subseteq \PP (S)$ denote a sublattice. It is called
\emph{decomposable} when there is a partition $S = S_1 \dot{\cup}\ts
S_2$ of $S$ into two non-empty subsets such that $\VV \subseteq \PP
(S_1) \times \PP (S_2)$. Moreover, $\VV$ is called \emph{reducible}
when a subset $U \nts \subset S$ exists such that $\VV$ is isomorphic
with a sublattice of $\PP (U)$. In this case, one may find a partition
$\{ A_1, A_2, \dots , A_r \} \in \PP (S)$, with $r < n$, such that
$\VV$ can be identified with a sublattice of $\PP \bigl(\{
1,2,\dots,r\} \bigr)$.

Clearly, we are primarily interested in lattices $\VV$ that are both
\emph{indecomposable} and \emph{irreducible}, which implies that
$\pmin, \pmax \in \VV$. The perhaps most important cases are $\PP (S)$
itself and $\II (S)$, the sublattice of \emph{ordered} or
\emph{interval partitions}. The latter are the partitions $\cA = \{
A_1, A_2, \dots , A_r \}$ such that all parts $A_i$ are contiguous
subsets of $S$, so
\[
     \cA \, = \, \big\{ \{1, \dots, k_1\},
      \{ k_1 + 1, \dots , k_2 \}, \dots ,
      \{ k_{r-1} + 1, \dots, n \} \big\} .
\]
Another interesting lattice consists of all non-crossing partitions.

Apart from our lattice $\VV$ for the set $S$, we will need the induced
lattices for all $\varnothing \ne U \nts \subseteq S$. They are defined
by
\[
      \VV (U) \, := \, \{ \cA|^{\pa}_{U}  \mid \cA \in \VV \} \ts ,
\]
where $\cA|^{\pa}_{U}$ denotes the restriction of $\cA$ to $U$. It
consists of all non-empty sets of the form $A_i \cap U$ as its
parts. With this notation, one has $\VV = \VV (S)$. We write $\lvert
\cA \rvert$ for the number of parts of a partition $\cA$, and $\supp
(\cA) := \bigcup_{i=1}^{\lvert \cA \rvert} A_{i}$ for the underlying
set that is partitioned.

In our setting, we also have a non-negative function
\[
       \varrho \! : \,  \PP (S) \longrightarrow \RR_{\geqslant 0} \ts ,
\]
whose values $\varrho (\cA)$ for $\cA \in \PP (S)$ have the meaning of
\emph{recombination rates} in the underlying continuous dynamical 
system we want to study.
Of particular relevance is the sublattice
\[
    \VV \, := \, \big\langle \cA \in \PP (S) \mid
    \varrho (\cA) > 0 \big\rangle \ts ,
\]
which is the smallest sublattice of $\PP (S)$ that contains all
partitions with strictly positive rates. To guarantee that $\VV$
is both irreducible and indecomposable, we will usually assume
that it contains all relevant partitions with two parts.

It will be vital to our later analysis that quantities defined on $\VV
(S)$ have consistent counterparts on all induced sublattices for
non-empty subsets $U\nts\subset S$. For the recombination rates, this
is done via
\begin{equation}\label{eq:rho-sub}
      \varrho^{U}\! (\cA) \, := 
     \!\sum_{\substack{\cB \in \VV (S) \\ \cB|^{\pa}_{U} = \cA}}
      \!\varrho^{S}\nts (\cB) \ts .
\end{equation}
This `top down' formula is consistent with the standard
marginalisation of probability vectors $q^{S}$ for the top system to
probability vectors $q^{U}$ for the subsystem defined by $U$; see
\cite[Sec.~5]{main} for a detailed account.

One quantity that requires a different approach is the \emph{decay
  rate} of a partition, which needs a `bottom up' recursion.  The
decay rates $\psi = \psi^{S}$ are defined recursively via those of the
subsystems for non-empty $U\nts \subseteq S$. One has
\cite{EB-ICM,main}
\begin{equation}\label{eq:psi-one}
     \psi^{U}\! (\pmax) \, = \sum_{\cA \ne \pmax} \varrho^{U}\! (\cA) \ts ,
\end{equation}
which implies $\psi^{U}\! (\pmax) = 0$ whenever $\lvert U
\rvert = 1$, together with the recursion
\begin{equation}\label{eq:psi-two}
    \psi^{U}\! (\cA) \, := \sum_{i=1}^{\lvert \cA \rvert}
     \psi^{A_i} (\pmax|^{\pa}_{A_{i}})
\end{equation}
for any partition $\cA = \{ A_1, A_2, \dots, A_{\lvert \cA \rvert} \}
\in \VV (U)$.  In particular, one always has 
$\psi^{U}\! (\pmin) = 0$. 

Functions such as $\varrho$ or $\psi$ can be viewed as elements of the
\emph{M\"{o}bius algebra} $\MM^{\pa}_{\RR} (\VV)$ for $\VV$ over the
field $\RR$. This is the free real vector space with basis $\{
\varepsilon^{\pa}_{\!\cA} \mid \cA \in \VV \}$, where
$\varepsilon^{\pa}_{\!\cA}$ is the function given by
$\varepsilon^{\pa}_{\!\cA} (\cB) := \delta^{\pa}_{\!\cA,\cB}$.  Then,
multiplication is defined by $ \varepsilon^{\pa}_{\!\cA} \cdot
\varepsilon^{\pa}_{\cB} = \varepsilon^{\pa}_{\!\cA} \,
\delta^{\pa}_{\!\cA,\cB}$; compare \cite[Sec.~IV.4.A]{Aigner} for
details.  More important to us is the real \emph{incidence algebra}
\cite[Ch.~IV.1]{Aigner}
\[
     \AAA^{\pa}_{\RR} (\VV) \, := \,
     \bigl\{ f \! : \, \VV \times \VV
     \longrightarrow \RR \mid
     f (\cA, \cB) = 0 \text{ whenever }
     \cA \not\preccurlyeq \cB \bigr\} .
\]
Here, addition and multiplication with a real number are defined
as usual, while multiplication $f * g$ is defined via the
(generally non-commutative) \emph{convolution}
\[
    \bigl( f * g \bigr) (\cA, \cB) \, = \!
    \sum_{\cA \preccurlyeq \udo{\cC}\preccurlyeq \cB}
    \! f(\cA,\cC) \, g (\cC,\cB) \ts ,
\]
where we follow \cite{Aigner} to mark the summation variable by
placing a dot underneath it.  Note that this gives $ \bigl( f * g
\bigr) (\cA, \cB) = 0$ for $\cA \not\preccurlyeq \cB$. Clearly,
$\delta (\cA,\cB) = \delta^{\pa}_{\!\cA,\cB}$ defines an element
$\delta\in\AAA^{\pa}_{\RR} (\VV)$ that is the neutral element for
multiplication, so $\delta\ts * f = f * \ts \delta = f$ for all $f\in
\AAA^{\pa}_{\RR} (\VV)$.

Recall that an element $f\in \AAA^{\pa}_{\RR} (\VV)$ is a \emph{unit}
(or invertible) when some $g\in \AAA^{\pa}_{\RR} (\VV)$ exists such
that $f*g = g * f = \delta$.  This happens if and only if $f (\cA,\cA)
\ne 0$ for all $\cA \in \VV$, and the inverse is then unique, hence
written as $f^{-1}$. The left inverse property gives $f^{-1} (\cA,\cA)
= 1/f(\cA,\cA)$ together with the recursion
\[
   f^{-1} (\cA,\cB) \, = \,
   \frac{-1}{f (\cB,\cB)}
    \sum_{\cA \preccurlyeq \udo{\cC}\prec \cB}
   f^{-1} (\cA,\cC) \, f(\cC,\cB)
\]
for $\cA\prec \cB$. A similar recursion follows from the right
inverse property. In the Appendix, we add a classic (but less
well-known) non-recursive formula \cite{SD}, which is sometimes of
advantage. An important unit is the \emph{zeta function} $\zeta$,
defined by $\zeta(\cA,\cB) = 1$ for $\cA\preccurlyeq \cB$ and by
$\zeta(\cA,\cB) = 0$ otherwise. Its inverse, $\mu = \zeta^{-1}$, is
the \emph{M\"{o}bius function} of the lattice, which plays a central
role in many (if not most) lattice-theoretic calculations; compare
\cite[Prop.~4.6]{Aigner}. It is given by $\mu (\cA,\cA ) =1$
for all $\cA\in\VV$ together with
\[
   \mu (\cA,\cB) \, = \, - \!
    \sum_{\cA \preccurlyeq \udo{\cC}\prec \cB} \!
   \mu (\cA,\cC)   \, =  \, - \!
    \sum_{\cA \prec \udo{\cC}\preccurlyeq \cB} \!
    \mu (\cC,\cB) 
\]
for $\cA \prec \cB$. It can often also be given in closed
form, which will be used below.

\section{Recombinaton equations}\label{sec:recall}
  
The ingredients for the formulation of the recombination ODE are a
finite set $S$ together with an irreducible and indecomposable lattice
$\VV$ of partitions of $S$, a general product space $X = \bigtim_{i\in
  S} \, X_{i}$ with all $X_{i}$ locally compact, the Banach space $\cM
(X)$ of finite Borel measures on $X$ with total variation norm $\|
. \|$, and the set of nonlinear \emph{recombinators} $R^{\pa}_{\!\cA}$
for $\cA\in\VV$, defined by
\[
       R^{\pa}_{\!\cA} (\omega) \, = \, 
       \frac{1}{\| \omega \|^{\lvert \cA \rvert -1}}
       \bigotimes_{i=1}^{\lvert \cA \rvert} 
       ( \pi^{\pa}_{\! A_{i}} .\ts \omega )
\]
for $\omega \in \cM (X)$. On the positive cone $\cM_{+} (X)$, the
recombinators satisfy $R^{\pa}_{\!\cA} R^{\pa}_{\cB} = R^{\pa}_{\!\cA
  \wedge \cB}$ for $\cA,\cB \in \VV$. Note that one can define a
matching order relation among them by saying that $R^{\pa}_{\!\cA}
\preccurlyeq R^{\pa}_{\cB}$ precisely when $R^{\pa}_{\!\cA}
R^{\pa}_{\cB} = R^{\pa}_{\!\cA}$. Consequently, one has
$R^{\pa}_{\!\cA} \preccurlyeq R^{\pa}_{\cB}$ if and only if
$\cA\preccurlyeq\cB$.

The recombination ODE now reads \cite{BB,MB,main}
\begin{equation}\label{eq:main-ODE}
     \dot{\omega}^{\pa}_{t} \, = \sum_{\cA\in\VV}
     \varrho (\cA) \, \bigl( R^{\pa}_{\!\cA} - \one \bigr)
       (\omega^{\pa}_{t})
\end{equation}
with non-negative numbers $\varrho (\cA) = \varrho^{S} (\cA)$ as
introduced above.  The forward flow preserves the positive cone of
$\cM (X)$ and, within it, also the norm of the initial condition,
$\omega^{\pa}_{0}$ say, see \cite{main} and references therein for
details. Since this is also the setting for the dynamics of
recombination, we restrict our attention to probability measures on
$X$ (denoted by $\cP (X)$) as initial conditions from now on.

The key to solve Eq.~\eqref{eq:main-ODE}, with $\omega^{\pa}_{0} \in
\cP (X)$, is the decomposition
\begin{equation}\label{eq:ansatz}
   \omega^{\pa}_{t} \, = \sum_{\cA \in \VV} 
   a^{\pa}_{t} (\cA) \, R^{\pa}_{\!\cA} (\omega^{\pa}_{0}) \ts ,
\end{equation}
which effectively separates the dependence on the initial condition
from the time evolution.  Here, the coefficients $a^{\pa}_{t}$
constitute a family of probability vectors on $\VV$ with initial
condition $a^{\pa}_{0} (\cA) = \delta (\cA,\pmax)$.  As a result of
\cite[Lemma 2 and Prop.~7]{main}, the induced ODE for the coefficients
then reads
\begin{equation}\label{eq:a-ODE}
    \dot{a}^{\pa}_{t} (\cA) \, = \, -\psi (\pmax) \, a^{\pa}_{t} (\cA)
    \; + \!\sum_{\cA\preccurlyeq\udo{\cB}\prec \pmax}\!
    \varrho (\cB)\, \prod_{i=1}^{\lvert \cB \rvert}
    a^{B_{i}}_{t} (\cA |^{\pa}_{B_{i}}) \ts ,
\end{equation}
where $a^{U}_{t}$ is the marginalised probability vector on the
sublattice $\VV (U)$. Note that this vector has a clear meaning in
terms of a (stochastic) partitioning process, as explained in detail
in \cite[Sec.~6]{main}.  Each $a^{U}_{t}$ satisfies an ODE of the same
type as in Eq.~\eqref{eq:a-ODE} for the subsystem defined by
$\varnothing\ne \nts U \subseteq S$; compare \cite[Sec.~5]{main} for
details on the important marginalisation consistency of our
equations. A step-by-step repetition of the results from
\cite{BB,MB,main} in our more general lattice-theoretic setting now
leads to the following fundamental result.

\begin{theorem}\label{thm:main}
  The Cauchy problem defined by the ODE \eqref{eq:main-ODE} with
  initial condition\/ $\omega^{\pa}_{0} \in \cM (X)$ has a unique
  solution.  The corresponding flow in forward time preserves the
  space\/ $\cP (X)$ of probability measures on\/ $X$. For any\/
  $\omega^{\pa}_{0} \in \cP (X)$, the solution is of the form
  \eqref{eq:ansatz}, where the coefficients\/ $a^{\pa}_{t}$ constitute
  probability vectors on\/ $\VV$ that are the unique solution of the
  nonlinear ODE system \eqref{eq:a-ODE} with initial condition\/
  $a^{\pa}_{0} (\cA) = \delta (\cA, \pmax)$. \qed
\end{theorem}

The main progress of \cite{main} in comparison to previous work is the
insight that, for generic values of the recombination rates $\varrho$,
an explicit recursive formula for the solution can be derived. In our
present setting, it can be formulated as follows. Let $\psi^{U}\! $ be
the rate functions defined above, and set
\begin{equation}\label{eq:theta-sum}
     a^{U}_{t} \! (\cA) \, := \sum_{\cB \succcurlyeq \cA} 
    \theta^{U}\! (\cA,\cB) \, \ee^{-\psi^{U} \! (\cB) \ts t}
\end{equation}
for $\cA \in \VV (U)$ and all non-empty $U\nts \subseteq S$, where
$\theta^{U}\!$ is an element of the incidence algebra
$\AAA^{\pa}_{\RR} (\VV (U))$. Now, let $\theta^{U} \! (\pmax,\pmax) =
1$ together with
\begin{equation}\label{eq:theta-init}
    \theta^{U} \! (\cA, \pmax) \, = \, - \!
    \sum_{\cA\preccurlyeq\udo{\cB}\prec\pmax}
    \!\theta^{U} \! (\cA,\cB)
\end{equation}
for any $\cA \prec \pmax$, which is just a reflection of the initial
conditions we need. Inserting Eq.~\eqref{eq:theta-sum} into the ODE
\eqref{eq:a-ODE} and observing the marginalisation consistency then
leads to the recursion
\begin{equation}\label{eq:rec}
   \theta^{U} \! (\cA,\cB) \, = \sum_{\cB\preccurlyeq\udo{\cC}\prec\pmax}
   \frac{\varrho^{U} \! (\cC)}{\psi^{U} \! (\pmax) - \psi^{U} \! (\cB)}
   \, \prod_{i=1}^{\lvert \cC \rvert} \theta^{C_{i}}
   (\cA |^{\pa}_{C_{i}}, \cB |^{\pa}_{C_{i}})
\end{equation}
for all $\cA\preccurlyeq\cB\prec \pmax$, provided that no denominator
vanishes and that the exponential functions in Eq.~\eqref{eq:theta-sum}
are linearly independent. The entire analysis from \cite{main} for the
lattice $\PP (S)$ remains valid for our irreducible and indecomposable
lattice $\VV$, and leads to the following result.

\begin{theorem}\label{thm:theta-rec}
  Let\/ $S$ be a finite set, and $\VV=\VV (S)$ an irreducible,
  indecomposable lattice of partitions of\/ $S$. Assume that the decay
  rates\/ $\psi^{U}\! (\cA)$ are recursively defined for\/
  $\varnothing\ne U\nts\subseteq S$ according to
  Eqs.~\eqref{eq:psi-one} and\/ \eqref{eq:psi-two}, with the induced
  recombination rates\/ $\varrho^{U}\!$ on subsystems being given by
  Eq.~\eqref{eq:rho-sub}. Assume further that the decay rates\/
  $\psi^{S}\nts (\cA)$ with\/ $\cA\in\VV$ are distinct.
   
  Then, the exponential ansatz \eqref{eq:theta-sum} solves the Cauchy
  problem of Eq.~\eqref{eq:a-ODE} with initial condition\/ $a^{U}_{t}
  \! (\cA) = \delta (\cA,\pmax)$ for all\/ $\varnothing\ne U\nts
  \subseteq S$ if and only if the coefficients\/ $\theta^{U}\! $ are
  recursively determined by Eq.~\eqref{eq:rec} together with\/
  $\theta^{U} \! (\pmax,\pmax) = 1$ and\/ Eq.~\eqref{eq:theta-init}.
\end{theorem}
\begin{proof}[Sketch of proof]
  Once again, this is a step-by-step transfer of the arguments for the
  lattice $\PP (S)$ from \cite{main} to $\VV \subset \PP
  (S)$. Irreducibility and indecomposability of $\VV$ ensure that the
  latter contains $\pmin$ and $\pmax$, which also implies the
  corresponding statement for $\VV (U)$ and all $U$ under
  consideration.

  In particular, given $\varnothing\ne U \nts \subseteq S$, the
  distinctness of all $\psi^{U} \! (\cB)$ with $\cB \in \VV (U)$
  follows from the assumed distinctness of the $\psi^{S}\nts (\cA)$
  with $\cA\in\VV$ by an application of \cite[Lemma~5]{main}, which
  holds in the generality we need here.  This allows to use the ansatz
  \eqref{eq:theta-sum} together with a comparison of coefficients to
  conclude as claimed, where all appearing denominators are non-zero
  as a consequence of the distinctness of the decay rates on all
  levels.
\end{proof}

The condition on the distinctness of the decay rates $\psi^{S}\nts
(\cA)$ with $\cA\in\VV$ covers the generic case (both topologically
and measure-theoretically). When looking at the excluded cases, one
realises that some degeneracies are `harmless' in the sense that the
solution formula \eqref{eq:theta-sum} remains valid (which also means
that the selected exponential functions still suffice), while there
are also `bad' degeneracies that lead to singularities in the
recursion and render this approach incomplete (due to the necessity of
further functions beyond simple exponentials; see \cite[Sec.~9 and
Appendix]{main}).  Below, we shall discuss this phenomenon in detail
for $\VV=\II (S)$, and later derive an alternative explanation in
terms of Jordan matrices for linear ODE systems. At this stage, we
recall \cite[Cor.~6]{main} and state it for our more general situation
as follows.
\begin{coro}\label{coro:extend}
  Let\/ $S$ and\/ $\VV$ be as in Theorem~$\ref{thm:theta-rec}$. Then,
  the generic recursive solution extends to all recombination rates\/
  $\varrho^{S}\nts (\cA)$, with\/ $\cA\in \VV$, such that\/ $\psi^{U} \!
  (\pmax) \ne \psi^{U} \! (\cB)$ holds for all\/ $\varnothing\ne U\nts
  \subseteq S$ and all\/ $\cB\in\VV (U)$.  \qed
\end{coro}

In \cite{main}, we also defined a linear reference system for $\VV =
\PP (S)$. The definition extends to any lattice $\VV$ that contains
$\pmax$ as
\begin{equation}\label{eq:chi-sum}
   \chi(\cA) \, = \!\sum_{\cB \not\in [\cA,\pmax]} \!\!  \varrho (\cB) 
   \quad \text{with} \quad
   [\cA,\pmax] := \{ \cC \in \VV \mid \cA \preccurlyeq 
        \cC \preccurlyeq \pmax \} \ts ,
\end{equation}
where it is understood that we sum over the partitions of $\VV$,
subject to the condition specified, and analogously for the subsystems
defined by the non-empty subsets $U\nts \subseteq S$. In particular,
one always has $\chi (\pmax) = \psi (\pmax)$.  Then, a simple
calculation leads to the following result.
\begin{fact}\label{fact:chi-diff}
  If\/ $\pmax \in \VV$, the decay rates\/ $\chi$ of
  Eq.~\eqref{eq:chi-sum} satisfy the relation
\[ 
    \chi (\pmax) - \chi (\cA) \; = 
    \sum_{\cA \preccurlyeq \udo{\cB} \prec \pmax}
    \varrho (\cB) \ts .
\]
Moreover, if\/ $\VV = \II (S)$ and if\/ $\varrho (\cA) > 0$ holds for
all\/ $\cA$ with two parts, one has
\[
     \chi (\pmax) - \chi (\cA) \, > \, 0
\]  
  for all\/ interval partitions\/ $\cA \ne \pmax$.   \qed
\end{fact}

Whenever $\lvert U \rvert \leqslant 3$, we had $\psi^{U}\! =
\chi^{U}$, and one has $\psi^{U}\! (\pmax) = \chi^{U}\! (\pmax)$ for
any $U\nts \ne \varnothing$, while no general relation was derived in
\cite{main}. To formulate one, we say that a partition $\cB$
\emph{splits} a set $V \subseteq U$ when the restriction $\cB
|^{\pa}_{V}$ has at least two parts. It is then clear what it means to
say that $\cB$ splits one or more parts of another partition $\cA$.
Note that $\cB$ can split a part of $\cA$ without satisfying $\cB
\preccurlyeq \cA$.

This notation allows us to rewrite Eq.~\eqref{eq:chi-sum}
for $\VV = \VV (S)$ as 
\[
   \chi (\cA) \; = \sum_{\cB\not\in [\cA,\pmax]} \! \varrho (\cB) \, =
     \sum_{\substack{\cB \text{ splits at least} \\ \text{one part of } \cA}} 
     \!\! \varrho (\cB) \; = \, \sum_{n\geqslant 1} \,
     \sum_{\substack{\cB \text{ splits $n$} \\ \text{parts of } \cA}}
     \! \varrho (\cB) \ts ,
\]
where all partition sums run over $\VV$. The analogous relation also
holds for all $\VV (U)$ with $\varnothing \ne U\nts \subset
S$. Observe that the last identity implies the relation
\[
    \chi^{U}\! (\pmax) \, = \! \sum_{\pmax \ne \udo{\cA} \in \VV (U)} 
    \! \varrho^{U}\! (\cA)
    \, = \! \sum_{\pmax \ne \udo{\cA} \in \VV (U)} 
     \sum_{\substack{\cB \in \VV  \\ \cB |^{\pa}_{U} = \cA}}
    \varrho (\cB) \, = 
    \sum_{\substack{\cB \in \VV  \\ \cB \text{ splits } U}}
     \varrho (\cB) \ts .
\]
Now, we can rewrite $\psi$ as
\begin{align}
   \psi (\cA) \; & = \,\sum_{i=1}^{\lvert \cA \rvert} 
      \psi^{A_i} (\pmax|^{\pa}_{A_{i}}) \; =
      \, \sum_{i=1}^{\lvert \cA \rvert} 
      \chi^{A_i} (\pmax|^{\pa}_{A_{i}}) \; =
      \, \sum_{i=1}^{\lvert \cA \rvert} \,
      \sum_{\substack{\cB \text{ splits} \\ \text{at least } A_i }}
      \! \varrho (\cB) \nonumber \\
      &= \, \sum_{i=1}^{\lvert \cA \rvert} \,\biggl( \,
      \sum_{\substack{\cB \text{ splits} \\ 
      \text{only } A_i}} \!\! \varrho (\cB)
      \;\; + \sum_{\substack{\cB \text{ splits } A_i \text{ and} \\ 
            \text{$A_j$ for one $j\ne i$} }} \! \! 
            \varrho (\cB)   \;\; +
            \sum_{\substack{\cB \text{ splits $3$ parts of} \\ 
                \cA , \text{ including } A_i}} \! \! \varrho (\cB)
            \;\; + \; \dots \, \biggr) \nonumber \\[2mm]
      & = \,  \sum_{n \geqslant 1} \; n \!\!
      \sum_{\substack{\cB \text{ splits $n$} \\ \text{parts of } \cA}}
          \! \! \varrho (\cB) \; = \, \sum_{\cB\in\VV} 
          \# (\cA,\cB) \, \varrho (\cB) \ts ,  \label{eq:psi-sum} 
\end{align}
where the second to last step follows from a simple counting argument
and $\# (\cA,\cB)$ is the number of parts of $\cA$ that are split by
$\cB$.  If we now compare the two calculations and recall that
$\varrho \geqslant 0$, the following result is immediate.

\begin{lemma}\label{lem:psi-chi-gen}
    Within a given lattice\/ $\VV$, the decay rates\/ $\psi$
    and\/ $\chi$ satisfy the relation
\[
     \psi (\cA) \, = \, \chi (\cA) \, + \, \sum_{n\geqslant 2} \,
     (n-1) \! \! \sum_{\substack{\cB \text{ \rm splits } n \\ 
        \text{\rm parts of } \cA}}
          \! \! \varrho (\cB) \ts .
\]    
In particular, one has\/ $\psi (\cA) \geqslant \chi (\cA)$ for all\/
$\cA \in \VV$.  \qed
\end{lemma}

Since $\psi (\pmax) = \chi (\pmax)$ whenever $\pmax \in \VV$,
Lemma~\ref{lem:psi-chi-gen} together with Fact~\ref{fact:chi-diff}
provides an alternative way to express the difference $\psi (\pmax) -
\psi (\cA)$ in this case as
\[
   \psi (\pmax) - \psi (\cA) \, = \!
    \sum_{\cA \preccurlyeq \udo{\cB} \prec \pmax} \!
    \varrho (\cB) \;\, - \; \sum_{n\geqslant 2} \,
     (n-1) \! \! \sum_{\substack{\cB \text{ \rm splits } n \\ 
        \text{\rm parts of } \cA}}
          \! \! \varrho (\cB) \ts .
\]
This formula clearly shows that, in general, $\psi (\pmax) - \psi
(\cA)$ for $\cA \ne \pmax$ can have either sign.

Let us now look more closely into the lattice $\VV = \II (S)$,
which is simpler than $\PP (S)$, but still displays most of
the relevant general phenomena.

\section{Interval partitions}\label{sec:int-part}

This section deals with the case that $\VV = \II(S)$, where $S$ is a
fixed finite set, for instance $S= \{ 1,2,\dots , n\}$.

\begin{lemma}\label{lem:psi-is-chi}
  Let\/ $S$ be as above, and assume that\/ $\varrho^{S}\nts (\cC) = 0$
  for all\/ $\cC\in\II(S)$ with\/ $\lvert \cC \rvert > 2$. Then, for
  any\/ $\varnothing \ne U \subseteq S$, one also has\/ $\varrho^{U}
  \! (\cB) = 0$ for all\/ $\cB \in \II(U)$ with\/ $\lvert \cB \rvert >
  2$. Moreover, $\psi^{U} \! (\cA) = \chi^{U} \! (\cA)$ holds for
  all\/ $\varnothing \ne U \nts \subseteq S$ and all\/ $\cA \in \II
  (U)$.
\end{lemma}

\begin{proof}
  Since the induced (or marginalised) rates for $\cB \in \II (U)$ are
  given by
\[
     \varrho^{U} \! (\cB) \, =  
     \sum_{\substack{\cA\in\II(S)\\ \cA |^{\pa}_{U} = \cB}}
     \varrho^{S} \nts (\cA) \ts ,
\]
the first claim is a simple consequence of the fact that, if $\cB \in
\II (U)$ with $\lvert \cB \rvert > 2$, any $\cA \in \II (S)$ with $\cA
|^{\pa}_{U} \! = \cB$ must have at least three parts, too, due to the
structure of the interval partitions.

When one considers the result of Lemma~\ref{lem:psi-chi-gen} for the
lattice $\VV = \II (U)$, one finds that each term in the difference
$\psi^{U}\!  (\cA) - \chi^{U}\! (\cA)$ is of the form $\varrho^{U}\!
(\cB)$ for some $\cB\in \II (U)$ with $\lvert \cB \rvert > 2$, hence
vanishes by the first part of our proof.
\end{proof}  

For the generic choice of the recombination rates on $\II(S)$,
our general solution has the form
\begin{equation}\label{eq:int-sol}
     a^{S}_{t} \nts (\cA) \, =  \sum_{\cB \succcurlyeq \cA}
     \theta^{S} \nts (\cA, \cB) \, \ee^{- \psi^{S} \nts\nts (\cB) \ts t} ,
\end{equation}
with $\theta^{S}$ determined recursively as explained above.  Let
$\mu^{S}$ denote the M\"obius function of the lattice $\II (S)$, which
is given by $\mu^{S} (\cA, \cB) = (-1)^{\lvert \cA \rvert - \lvert \cB
  \rvert}$ for $\cA \preccurlyeq \cB$ and $\mu^{S} (\cA, \cB) = 0$
otherwise. The definition of $\mu^{U}$ for $\II (U)$ with $\varnothing
\ne U\nts \subset S$ is analogous. These functions satisfy the
following product identity.
\begin{fact}\label{fact:product}
  Let\/ $\cA, \cB, \cC \in \II (S)$ with\/ $\cA \preccurlyeq \cB
  \preccurlyeq \cC$. Then, one has
\[
   \mu^{S} (\cA, \cB) \, = \, \prod_{i=1}^{r}
   \mu^{C_{i}} \bigl( \cA |^{\pa}_{C_{i}} , \cB |^{\pa}_{C_{i}} \bigr) ,
\]
   where\/ $r=\lvert \cC \rvert$, so\/ $\cC = \{ C_{1}, C_{2}, 
   \dots , C_{r}\}$.
\end{fact}
\begin{proof}
Since $\cA \preccurlyeq \cC$, we certainly have $\lvert \cA \rvert =
\sum_{i=1}^{r} \big| \cA |^{\pa}_{C_{i}} \big|$, and
analogously for $\cB$. Now, one simply calculates
\[
\begin{split}
  \prod_{i=1}^{r}
  \mu^{C_{i}} \bigl( \cA |^{\pa}_{C_{i}} , \cB |^{\pa}_{C_{i}} \bigr)
  \, & = \, \prod_{i=1}^{r}
  (-1)_{\phantom{\hat{I}}}^{\big| \cA |^{\pa}_{C_{i}} \big| - 
                   \big| \cB |^{\pa}_{C_{i}} \big|}
  = \, (-1)_{\phantom{\hat{I}}}^{\sum_{i=1}^{r} 
       \big| \cA |^{\pa}_{C_{i}} \big| - 
       \big| \cB |^{\pa}_{C_{i}} \big| } \\
  & = \, (-1)_{\phantom{I}}^{\lvert \cA \rvert - \lvert \cB \rvert}
  = \, \mu^{S} (\cA, \cB) \ts ,
\end{split}
\]
which establishes the product formula.
\end{proof}

With this preparation, one finds the following simplification
of Eq.~\eqref{eq:int-sol}.

\begin{prop}\label{prop:reduction-to-linear}
  Under the assumptions of Lemma~$\ref{lem:psi-is-chi}$, where\/
  $\varrho (\cC) = 0$ for all\/ $\cC \in \II (S)$ with\/ $\lvert \cC
  \rvert > 2$, the solution formula \eqref{eq:int-sol} simplifies to
\[
     a^{S}_{t} \nts (\cA) \, = \sum_{\cB \succcurlyeq \cA}
     \mu^{S} \nts (\cA, \cB) \, \ee^{-\chi^{S} \nts (\cB) \ts t}
\]   
with\/ $\chi^{S}$ as above in Eq.~\eqref{eq:chi-sum}. Moreover, this
formula for\/ $a^{S}_{t}$ holds for all remaining choices of the
recombination rates.
\end{prop}

\begin{proof}
  In view of Lemma~\ref{lem:psi-is-chi}, it remains to prove that
  $\theta^{S}\! = \mu^{S}$.  This will be done inductively again, so
  we will show that $\theta^{U}\! = \mu^{U}$ holds for all non-empty
  $U \nts \subseteq S$.

  When $\lvert U \rvert = 1$, one has $\II (U) = \{ \pmax \}$ and
  $\theta^{U}\! (\pmax,\pmax) = 1 = \mu^{U}\! (\pmax,
  \pmax)$. Likewise, when $\lvert U \rvert = 2$, the corresponding
  lattice is $\II (U) = \{ \pmin, \pmax\}$. It is clear that
  $\theta^{U}\! (\pmin, \pmin) = \theta^{U}\! (\pmax, \pmax) = 1$, and
  one easily checks that $\theta^{U}\! (\pmin, \pmax) = -1$ together
  with $\theta^{U}\! (\pmax, \pmin) = 0$, so $\theta^{U} \! = \mu^{U}$
  in this case, too.

  Now, let us assume our claim to be true for all finite sets $U$ with
  cardinality $\lvert U \rvert \leqslant n$, and consider a set $S$
  with $\lvert S \rvert = n+1$.  We know that $\theta^{S}\nts
  (\cA,\cB) = 0$ whenever $\cA \not\preccurlyeq \cB$, so let $\cA,\cB
  \in \II (S)$ with $\cA \preccurlyeq \cB$ and assume first that $\cB
  \prec \pmax$ (we will deal with the remaining case later). Now, our
  recursion together with the induction assumption and
  Lemma~\ref{lem:psi-is-chi} gives us
\[
\begin{split}
   \theta^{S}\nts (\cA, \cB) \, & = 
      \sum_{\cB \preccurlyeq \udo{\cC} \prec \pmax}
      \frac{\varrho^{S}\nts (\cC)}{\psi^{S}\nts (\pmax) -
      \psi^{S}\nts (\cB)}\,  \prod_{i=1}^{\lvert \cC \rvert} \theta^{C_i} 
      (\cA |^{\pa}_{C_i} , \cB |^{\pa}_{C_i} ) \\
   & = \sum_{\cB \preccurlyeq \udo{\cC} \prec \pmax}
       \frac{\varrho^{S}\nts (\cC)}{\chi^{S}\nts (\pmax) - 
       \chi^{S}\nts (\cB)}\, \prod_{i=1}^{\lvert \cC \rvert} \mu^{C_i} 
       (\cA |^{\pa}_{C_i} , \cB |^{\pa}_{C_i} ) \\[1mm]
   & = \sum_{\cB \preccurlyeq \udo{\cC} \prec \pmax}
   \frac{\varrho^{S}\nts (\cC)}{\chi^{S}\nts (\pmax) - 
   \chi^{S}\nts (\cB)}\,
   \mu^{S}\nts (\cA, \cB) \, = \, \frac{\mu^{S}\nts (\cA, \cB)}
   {\chi^{S}\nts (\pmax) - \chi^{S}\nts (\cB)}
   \sum_{\cB \preccurlyeq \udo{\cC} \prec \pmax} \varrho^{S}\nts (\cC) \ts ,
\end{split}
\]
where the second line follows because $\cC \prec \pmax$ implies that
all parts $C_i$ satisfy $\lvert C_i \rvert \leqslant n$, while the
first step in the last line is a consequence of
Fact~\ref{fact:product}.

Now, using the assumption and Lemma~\ref{lem:psi-is-chi}, the last sum
evaluates as
\[
   \sum_{\cB \preccurlyeq \udo{\cC} \prec \pmax} \varrho^{S}\nts (\cC)
   \; = \, \sum_{\cC \ne \pmax} \varrho^{S}\nts (\cC) \; - \!
   \sum_{\cD \not \in [\cB, \pmax]} \varrho^{S}\nts (\cD) \, = \,
   \chi^{S}\nts (\pmax) - \chi^{S}\nts (\cB) \ts ,
\]
wherefore we get $\theta^{S}\nts (\cA,\cB) = \mu^{S}\nts (\cA,\cB)$
for all $\cA \preccurlyeq \cB \prec \pmax$. Finally, since
$\theta^{S}\nts (\pmax, \pmax) = 1$ is automatic, we consider an
arbitrary $\cA\prec \pmax$ and find
\[
    \theta^{S}\nts (\cA, \pmax) \, = \,
     - \! \sum_{\cA\preccurlyeq \udo{\cC}\prec \pmax}
    \theta^{S}\nts (\cA, \cC) \, = 
     - \! \sum_{\cA\preccurlyeq \udo{\cC}\prec \pmax}
    \mu^{S}\nts (\cA,\cC)
\]
from our previous argument. This finally also gives $\theta^{S}\nts
(\cA, \pmax) = \mu^{S}\nts (\cA, \pmax)$ from the M\"obius inversion
formula.

The final claim is clear, because the solution now extends to the
non-generic cases as well.
\end{proof}

This result explicitly shows the compatibility of the general solution
with the special case from \cite{BB,MB}, which was also mentioned in
\cite{main}. What is more, it also shows that the linearity situation
is slightly more general in the sense that one can still allow some
recombination rates beyond the two parent case to be positive. Let
us look into this point in more detail.

\section{Systems with up to four sites}\label{sec:examples}

When we deal with $\II (S)$ for any non-empty finite set $S$ with up
to $3$ elements, we know that
\[
    \psi^{S} \, = \, \chi^{S}
    \quad \text{together with} \quad
    \theta^{S} \, = \; \mu^{S}
    \quad \text{and} \quad
    \eta^{S} \, = \; \zeta^{S} .
\]
This means that the general solution is given by
\begin{equation}\label{eq:simple-sol}
     a^{S}_{t} (\cA) \, = \sum_{\cB \succcurlyeq \cA}
     \mu^{S} (\cA, \cB) \, \ee^{-\chi^{S} (\cB) \, t}
\end{equation}
for the initial condition $a^{S}_{0} (\cA) = \delta (\cA,\pmax)$,
irrespective of the values of the rates $\varrho$. For $\lvert S
\rvert = 1$, this gives $a^{\pa}_{t}\equiv 1$ (we drop the upper index
here), while $\lvert S \rvert = 2$, meaning $\II (S) = \{ \pmax, \pmin
\}$, results in $\chi (\pmax) = \varrho (\pmin)$ and thus in the
coefficient functions
\[
     a^{\pa}_{t} (\pmax) \, = \, \ee^{-\varrho (\pmin) \ts t}
     \quad \text{and} \quad
     a^{\pa}_{t} (\pmin) \, = \, 1 - \ee^{-\varrho (\pmin) \ts t} .
\]
For $S=\{1,2,3\}$, one has $\pmax = (123)$, $\pmin=(1|2|3)$ and decay
rates
\[
\begin{array}{c | c}
    \cA  & \chi \\ \hline
    (123)^{\vphantom{\hat{I}}}
         &  \varrho(1|23) + \varrho(12|3) + \varrho(1|2|3) \\
    (1|23) & \varrho(12|3) + \varrho(1|2|3) \\
    (12|3) & \varrho(1|23) + \varrho(1|2|3) \\
    (1|2|3) & 0
\end{array}    
\]
The coefficient functions now read
\[
\begin{split}
   a^{\pa}_{t} (\pmax) \, & = \, \ee^{-\chi(\pmax)\ts t} , \\
   a^{\pa}_{t} (1|23) \, &  = \, \ee^{-\chi(1|23) \ts t} 
                       - \ee^{-\chi(\pmax) \ts t} , \\
   a^{\pa}_{t} (12|3) \, &  = \, \ee^{-\chi(12|3) \ts t} 
                       - \ee^{-\chi(\pmax) \ts t} , \\
   a^{\pa}_{t} (\pmin) \, &  = \, 1 - \ee^{-\chi(1|23) \ts t} 
                   -\ee^{-\chi(12|3) \ts t} + \ee^{-\chi(\pmax) \ts t} . 
\end{split}
\]
Here, one can easily check that they indeed form a probability vector.
Note that formula \eqref{eq:simple-sol} also holds for $\PP (S)$, then
with the M\"{o}bius function $\mu$ and decay rate $\chi$ of $\PP (S)$;
see \cite{main} for more. Note that the lattice $\PP (S)$ contains one
additional partition, namely $(13|2)$, in comparison to $\II (S)$. Up
to this point, it is obvious that the probability vector $a^{\pa}_{t}$
also satisfies a linear ODE system --- a point of view we will look at
more closely in Section~\ref{sec:outlook} below.  \medskip

For $S = \{ 1,2,3,4 \}$, things start to get more complicated. 
The lattice $\II (S)$, in obvious notation, is given by
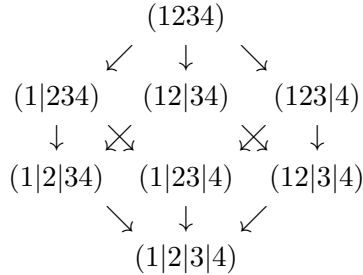
\[
\begin{array}{c@{}c@{}c@{}c@{}c}
   &  & (1 2 3 4) & \\
 & \swarrow & \downarrow & \searrow & \\
(1 | 2 3 4) & & (12|34) & & (123|4) \\
   \downarrow & \searrow \hspace*{-2.5ex} \swarrow & & 
     \swarrow \hspace*{-2.2ex} \searrow & \downarrow \\
(1|2|34) & & (1|23|4) & & (12|3|4) \\
 & \searrow & \downarrow & \swarrow  & \\
 & & (1|2|3|4) & &
\end{array}   
\]
where an arrow points towards the next refined partition.  An
application of Lemma~\ref{lem:psi-chi-gen} to $\VV = \II (S)$ gives
the rate function as
\begin{equation}\label{eq:psi-four-sites}
   \psi (\cA) \, = \, \chi (\cA) +
   \bigl(\varrho (1|23|4) + \varrho (\pmin) \bigr)
   \delta \bigl(\cA, (12|34) \bigr) ,
\end{equation}
with $\chi$ as in Eq.~\eqref{eq:chi-sum}.  An application of the
recursion formula now results in the following expression for
$\theta$, which we show in (upper triangular) matrix form for
convenience,
\begin{equation}\label{eq:theta-4-sites}
\begin{array}{c|cccccccc}
  & \rotatebox[origin=c]{90}{$\,(1|2|3|4)\,$}
  & \rotatebox[origin=c]{90}{$(12|3|4)$}
  & \rotatebox[origin=c]{90}{$(1|23|4)$} 
  & \rotatebox[origin=c]{90}{$(1|2|34)$}
  & \rotatebox[origin=c]{90}{$(123|4)$} 
  & \rotatebox[origin=c]{90}{$(12|34)$} 
  & \rotatebox[origin=c]{90}{$(1|234)$}
  & \rotatebox[origin=c]{90}{$(1234)$} \\
\hline
(1|2|3|4)^{\vphantom{\hat{I}}} 
          & \,1\, & -1 & -1 & -1 &  1 &  x &  1 & -x \\
(12|3|4)  & 0 &  1 &  0 &  0 & -1 & -x &  0 &  x \\
(1|23|4)  & 0 &  0 &  1 &  0 & -1 &  0 & -1 &  1 \\
(1|2|34)  & 0 &  0 &  0 &  1 &  0 & -x & -1 &  x \\
(123|4)   & 0 &  0 &  0 &  0 &  1 &  0 &  0 & -1 \\
(12|34)   & 0 &  0 &  0 &  0 &  0 &  x &  0 & -x \\
(1|234)   & 0 &  0 &  0 &  0 &  0 &  0 &  1 & -1 \\
(1234)    & 0 &  0 &  0 &  0 &  0 &  0 &  0 &  1
\end{array}
\end{equation}
Here, one has 
\[
   x \, = \, \frac{\varrho (12|34)}{ \psi(\pmax) - \psi(12|34)}
   \quad \text{with} \quad
    \psi(\pmax) - \psi(12|34) \, = \,
    \varrho (12|34) - \bigl( \varrho (1|23|4)
     + \varrho (1|2|3|4) \bigr) .
\]
This clearly shows that $x$ can have either sign in general.  When the
denominator vanishes, we are outside the validity of the simple
solution formula; see \cite[Ch.~9]{main} for details. We shall come
back to this case below.

Before we expand on the solution, let us stick to the case that $x\ne
0$, where $\theta$ is invertible.  The inverse function $\eta$ is
still surprisingly simple, and given by $\eta (\cA, \cB) = \zeta (\cA,
\cB)$ unless $\cA = \cB = (12|34)$, in which case one has
\[ 
   \eta \bigl( (12|34),(12|34)\bigr)
   = \, \frac{1}{x} \, = \,
   \frac{\psi(\pmax) - \psi(12|34)}{\varrho(12|34)} 
   \, = \, 1 - \frac{\varrho(\pmin) + \varrho(1|23|4)}
   {\varrho(12|34)}\ts .
\]
So, one sees that $\theta = \mu$ (and thus $\eta = \zeta$)
happens precisely when $\psi = \chi$, which is equivalent
with the condition $\varrho (\pmin) = \varrho(1|23|4) =0$.

\begin{coro}\label{coro:all-clear}
  In the case of the lattice\/ $\II (\{ 1,2,3,4\})$, we get\/ $\psi
  (\cA) = \chi (\cA)$ for\/ \emph{all} $\cA$ precisely when\/ $\varrho
  (\pmin) = \varrho(1|23|4) =0$. Then, one also has\/ $\theta = \mu$
  together with\/ $\eta = \zeta$, and the linear solution formula
  applies to\/ \emph{all} choices of the remaining recombination
  rates.  \qed
\end{coro}

Now, since $4$ is the smallest system size where $\theta$ can become
non-invertible (which happens when $x=0$), we can go one step back and
write down the ODE system for the functions $a^{\pa}_{t} (\cA)$ on
$\II (S)$ explicitly as
\begin{equation}\label{eq:4-sites}
   \dot{a}^{\pa}_{t} (\cA) \, = \, - \psi (\pmax) \, a^{\pa}_{t} (\cA)
   \; + \sum_{\cA \preccurlyeq \udo{\cB} \prec \pmax}
   \mu (\cA, \cB) \sum_{\cB \preccurlyeq \udo{\cC} \prec \pmax}
   \! \varrho (\cC)  \, \ee^{-\psi (\cB) \ts t} ,
\end{equation}
where we have again dropped the upper index. This ODE follows from
\cite[Prop.~7]{main} by restriction to the lattice $\II (S)$ and the
insertion of the explicit solution for the possible subsystems, which
have at most $3$ sites and are thus given as above. Alternatively,
one can use Eq.~\eqref{eq:a-ODE} for a direct derivation. With
\cite[Lemma~8]{main}, see also Fact~\ref{fact:ODE} below, 
the solution of Eq.~\eqref{eq:4-sites} is
\[
   a^{\pa}_{t} (\cA) \, = \, a^{\pa}_{0} (\cA) \, 
   \ee^{-\psi(\pmax) \ts t} \; + \sum_{\cA \preccurlyeq \udo{\cB} \prec \pmax}
   \mu (\cA, \cB) \sum_{\cB \preccurlyeq \udo{\cC} \prec \pmax}
   \varrho (\cC) \,
   E^{}_{0} \bigl(\psi(\pmax),\psi(\cB); t \bigr)
\]
with the function
\begin{equation}\label{eq:E-fun}
   E^{}_{0} (\alpha, \beta; t) \, = \,
   \begin{cases} t \, \ee^{-\alpha t} , & 
      \text{if $\alpha = \beta$}, \\
   \frac{1}{\alpha - \beta} 
      \bigl(\ee^{-\beta t} - \ee^{-\alpha t} \bigr), &
       \text{otherwise}.
   \end{cases}
\end{equation}
Note that $E^{\pa}_{0}$ is continuous in the parameters, and symmetric
under the exchange $\alpha \leftrightarrow \beta$.

Now, observing that $\psi (\pmax) - \psi (\cA) = \chi (\pmax) - \chi
(\cA)$ for all $\cA \ne (12|34)$ by Eq.~\eqref{eq:psi-four-sites} and
using the first identity of Fact~\ref{fact:chi-diff}, the solution
simplifies to
\[
\begin{split}
  a^{\pa}_{t} (\cA) \, = \; \, & a^{\pa}_{0} (\cA) \, 
   \ee^{-\psi(\pmax) \ts t} \; + \!\!
   \sum_{\substack{\cA \preccurlyeq \udo{\cB} \prec \pmax \\
          \cB \ne (12|34)}} \!\! \mu (\cA, \cB) \, 
    \bigl( \ee^{-\psi (\cB) \ts t} - \ee^{-\psi (\pmax)\ts t} \bigr) \\
    & + \mu \bigl( \cA, (12|34)\bigr) \,
   \varrho(12|34) \, E^{\pa}_{0} (\psi (\pmax), \psi(12|34); t) \ts ,
\end{split}
\]
with $\varrho(12|34) = \bigl( \psi(\pmax) - \psi(12|34) \bigr)
+  \varrho(1|23|4) + \varrho(\pmin)$.  
The initial condition relevant to us is $a^{\pa}_{0}
(\cA) = \delta (\cA,\pmax)$.  Together with the upper M\"{o}bius
summation, $\sum_{\cA \preccurlyeq\udo{\cB} \preccurlyeq \pmax} \mu
(\cA,\cB) = \delta (\cA,\pmax)$, the last equation then further
simplifies to
\begin{equation}\label{eq:sol-4-sites}
    a^{\pa}_{t} (\cA) \,  =   \sum_{\udo{\cB} \succcurlyeq\cA}
    \! \mu (\cA,\cB) \, \ee^{-\psi (\cB) \ts t} \;
      + \; \mu\bigl(\cA, (12|34)\bigr)
    \bigl( \varrho (\pmin) + \varrho (1|23|4) \bigr)
   E^{\pa}_{0} (\psi (\pmax), \psi(12|34); t) \ts .
\end{equation}
Here, when $\varrho(1|23|4) + \varrho(\pmin) \ne 0$, the last factor
becomes $ t \, \ee^{-\psi (\pmax) \ts t}$ (via l'Hospital's rule, in
agreement with the definition of $E^{\pa}_{0}$)
when $\varrho(12|34) \to \varrho(1|23|4) + \varrho(\pmin)$, which
means $\psi(12|34) \to\psi(\pmax)$. This shows the appearance of a new
function in the expansion when we hit one of the essential
singularities of the recursive solution described earlier.  Note that
the full solution can still be extracted from the $\theta$-matrix, if one
starts with a generic case and then rewrites the contributions with a
prefactor $\pm x$ in terms of the function $E^{\pa}_{0}$.

In general, when $\varrho(1|23|4) = \varrho(\pmin) = 0$, we are back
in the situation of Corollary~\ref{coro:all-clear}, and thus in the
realm of the linear scheme. Note that no restrictions occur for the
values of $\varrho$ on $(1|2|34)$ or $(12|3|4)$, so that the linear
scheme actually slightly extends beyond $\II_{2} (S)$. In any case,
one has $a^{\pa}_{t} (\cA) \geqslant 0$ for all $\cA\in\II(S)$ and all
$t\geqslant 0$, due to the preservation of positivity under the
forward flow, together with $\sum_{\cA\in\II(S)} a^{\pa}_{t} (\cA)=1$
for all $t\geqslant 0$, as a result of $\sum_{\cA\in\II(S)}
\mu\bigl(\cA,(12|34)\bigr) = 0$ and $\sum_{\udo{\cA} \preccurlyeq\cB}
\mu (\cA,\cB) = \delta(\pmin,\cB)$ together with $\psi(\pmin) = 0$,
which is nothing but a direct verification of the norm preservation
mentioned in the Introduction.

\begin{coro}\label{coro:sol-4-sites}
  The solution of the recombination equation \eqref{eq:main-ODE} on\/
  $\cM (X)$, for\/ $S=\{ 1,2,3,4\}$ with the lattice\/ $\VV = \II (S)$
  and with a probability measure\/ $\omega^{\pa}_{0} \in \cP (X)$ as
  initial condition, is given by
\[
      \omega^{}_{t} \, =
      \sum_{\cA\in\II (S)} a^{\pa}_{t} (\cA) \, 
      R^{\pa}_{\nts\cA} (\omega^{\pa}_{0}) \ts ,
\]      
   with the convex coefficients\/ $a^{\pa}_{t} (\cA)$ of
    Eq.~\eqref{eq:sol-4-sites}.  \qed
\end{coro}

\section{Recombination with interval partitions
 for five sites}\label{sec:five-sites}

Let us turn to the case of five sites, $S=\{1,2,3,4,5\}$. Here, as a
result of Lemma~\ref{lem:psi-chi-gen}, the decay rates are related by
\[
    \psi (\cA) \, = \, \chi (\cA) \; + \!
    \sum_{\substack{\cB \text{ splits two} \\ \text{parts of } \cA}} 
    \! \varrho (\cB) \ts ,
\]
as a partition $\cA \in \II (S)$ can have at most two splittable
parts.  In particular, $\psi (\cA) = \chi (\cA)$ for any $\cA$ with at
most one non-singleton part. For the other partitions, we get
\begin{equation}\label{eq:rho-tab}
  \begin{array}{c | c}
  \cA & \psi (\cA) - \chi (\cA) \\ \hline
  (12|345)^{\vphantom{\hat{I}}}  
      & \varrho(1|23|45) + \varrho(1|234|5) + 
        \varrho(1|2|3|45) + \varrho(1|2|34|5) +
        \varrho(1|23|4|5) + \varrho(\pmin) \\
  (123|45)  & \varrho(12|34|5) + \varrho(1|234|5) + 
        \varrho(12|3|4|5) + \varrho(1|23|4|5) +
        \varrho(1|2|34|5) + \varrho(\pmin)\\
  (1|23|45) & \varrho(12|34|5) + \varrho(1|2|34|5) + 
              \varrho(12|3|4|5) + \varrho(\pmin) \; \\
  (12|3|45) & \varrho(1|234|5) + \varrho(1|2|34|5) + 
              \varrho(1|23|4|5) + \varrho(\pmin) \; \\
  (12|34|5) & \varrho(1|23|45) + \varrho(1|23|4|5) + 
              \varrho(1|2|3|45) + \varrho(\pmin) \;
  \end{array}
\end{equation}
In particular, we get $\psi = \chi$ once again for a few more cases
than those covered in Lemma~\ref{lem:psi-is-chi}, because no
conditions emerge for the rates $\varrho(1|2|345)$, $\varrho(12|3|45)$
and $\varrho(123|4|5)$.

\begin{coro}\label{coro:all-clear-five}
  In the case of the lattice\/ $\II (\{ 1,2,3,4,5\})$, we get\/ $\psi
  (\cA) = \chi (\cA)$ for\/ \emph{all} $\cA$ precisely when the eight
  recombination rates vanish that appear in
  Eq.~\eqref{eq:rho-tab}. Then, one also has\/ $\theta = \mu$ together
  with\/ $\eta = \zeta$, and the linear solution formula applies to\/
  \emph{all} choices of the remaining recombination rates.  \qed
\end{coro}

\begin{table}[ht]
\begin{footnotesize}
\begin{tabular}{c| c @{\;\;} c @{\:\:} c @{\;\:} c @{\;\:} 
   c @{\;\:} c @{\:\:} c @{\:} c @{\:\:} c @{\:\:} 
   c @{\:\:} c @{\;\:} c @{\:\,} c @{\:\:} c @{\:\:} c @{\:\:} c}
  & \rotatebox[origin=c]{90}{$\;(1|2|3|4|5)\;$}
  & \rotatebox[origin=c]{90}{$(12|3|4|5)$}
  & \rotatebox[origin=c]{90}{$(1|23|4|5)$} 
  & \rotatebox[origin=c]{90}{$(1|2|34|5)$}
  & \rotatebox[origin=c]{90}{$(1|2|3|45)$}
  & \rotatebox[origin=c]{90}{$(123|4|5)$} 
  & \rotatebox[origin=c]{90}{$(12|34|5)$} 
  & \rotatebox[origin=c]{90}{$(12|3|45)$}
  & \rotatebox[origin=c]{90}{$(1|234|5)$}
  & \rotatebox[origin=c]{90}{$(1|23|45)$}
  & \rotatebox[origin=c]{90}{$(1|2|345)$} 
  & \rotatebox[origin=c]{90}{$(1234|5)$} 
  & \rotatebox[origin=c]{90}{$(123|45)$}
  & \rotatebox[origin=c]{90}{$(12|345)$}
  & \rotatebox[origin=c]{90}{$(1|2345)$} 
  & \rotatebox[origin=c]{90}{$(12345)$} \\
\hline
$(1|2|3|4|5)^{\vphantom{\hat{I}}}$
   & ${\, 1\,}$ & $-1$ & $-1$ & $-1$ & $-1$ & 1 & $x^{\pa}_{1}$ & $x^{\pa}_{2}$ 
   & 1 & $x^{\pa}_{3}$ & 1 & $-x^{\pa}_{1}$  & $-x^{\pa}_{4}$  & $-x^{\pa}_{5}$  
   & $-x^{\pa}_{3}$ & $x^{\pa}_{5} \!+\! x^{\pa}_{4} \!-\! x^{\pa}_{2}$ \\
$(12|3|4|5)$
   & 0 & 1 & 0 & 0 & 0 & $-1$ & $-x^{\pa}_{1}$ & $-x^{\pa}_{2}$ 
   & 0 & 0 & 0 & $x^{\pa}_{1}$  & $x^{\pa}_{4}$  & $x^{\pa}_{5}$  & 0 
   & $x^{\pa}_{2} \!-\! x^{\pa}_{4} \!-\! x^{\pa}_{5}$ \\
$(1|23|4|5)$
   & 0 & 0 & 1 & 0 & 0 & $-1$ & 0 & 0 & $-1$ & $-x^{\pa}_{3}$ & 0 & 1 
   & $x^{\pa}_{4}$  & 0 & $x^{\pa}_{3}$   & $-x^{\pa}_{4}$ \\
$(1|2|34|5)$
   & 0 & 0 & 0 & 1 & 0 & 0 & $-x^{\pa}_{1}$ & 0 & $-1$ & 0 & $-1$ 
   & $x^{\pa}_{1}$ & 0 & $x^{\pa}_{5}$  & 1 & $-x^{\pa}_{5}$ \\
$(1|2|3|45)$
   & 0 &  0 &  0 &  0 &  1 & 0 & 0 & $-x^{\pa}_{2}$ & 0 
   & $-x^{\pa}_{3}$ & $-1$ & 0 & $x^{\pa}_{4}$ & $x^{\pa}_{5}$  
   & $x^{\pa}_{3}$ & $x^{\pa}_{2} \!-\! x^{\pa}_{4} \!-\! x^{\pa}_{5}$ \\
$(123|4|5)$
   & 0 & 0 & 0 & 0 & 0 & 1 & 0 & 0 & 0 & 0 & 0 & $-1$
   & $-x^{\pa}_{4}$ & 0 & 0 & $x^{\pa}_{4}$  \\
$(12|34|5)$
   & 0 & 0 & 0 & 0 & 0 & 0 &  $x^{\pa}_{1}$ & 0 & 0 & 0 & 0 
   & $-x^{\pa}_{1}$ & 0 & $-x^{\pa}_{5}$ & 0 & $x^{\pa}_{5}$ \\
$(12|3|45)$
   & 0 & 0 & 0 & 0 & 0 & 0 & 0 & $x^{\pa}_{2}$ & 0 & 0 & 0 & 0 
   & $-x^{\pa}_{4}$ & $-x^{\pa}_{5}$ & 0 
   & $x^{\pa}_{5} \!+\! x^{\pa}_{4} \!-\! x^{\pa}_{2}$ \\
$(1|234|5)$
   & 0 & 0 & 0 & 0 & 0 & 0 & 0 & 0 & 1 & 0 & 0 & $-1$ 
   & 0 & 0 & $-1$ & 1 \\
$(1|23|45)$
   & 0 & 0 & 0 & 0 & 0 & 0 & 0 &  0 & 0 & $x^{\pa}_{3}$ 
   & 0 & 0 & $-x^{\pa}_{4}$ & 0 & $-x^{\pa}_{3}$ & $x^{\pa}_{4}$ \\
$(1|2|345)$
   & 0 & 0 & 0 & 0 & 0 & 0 & 0 & 0 & 0 & 0 & 1 & 0 
   & 0 & $-x^{\pa}_{5}$ & $-1$ & $x^{\pa}_{5}$ \\
$(1234|5)$
   & 0 & 0 & 0 & 0 & 0 & 0 & 0 & 0 & 0 & 0 & 0 & 1 & 0 & 0 & 0 & $-1$ \\
$(123|45)$
   & 0 & 0 & 0 & 0 & 0 & 0 & 0 & 0 & 0 & 0 & 0 & 0 & $x^{\pa}_{4}$  
   & 0 & 0 & $-x^{\pa}_{4}$ \\
$(12|345)$
   & 0 & 0 & 0 & 0 & 0 & 0 & 0 & 0 & 0 & 0 & 0 & 0 & 0 
   & $x^{\pa}_{5}$ & 0 & $-x^{\pa}_{5}$ \\
$(1|2345)$
   & 0 & 0 & 0 & 0 & 0 & 0 & 0 & 0 & 0 & 0 & 0 & 0 & 0 & 0 & 1 & $-1$ \\
$(12345)$
   & 0 & 0 & 0 & 0 & 0 & 0 & 0 & 0 & 0 & 0 & 0 & 0 & 0 & 0 & 0 & 1
\end{tabular}\bigskip

\caption{The $\theta$-matrix for the interval partitions of five 
sites.}\label{tab:matrix}
\end{footnotesize}
\end{table}

The recursive formula for $\theta$ leads to the matrix shown in
Table~\ref{tab:matrix}. Here, we used the shorthands
\[
   x^{\pa}_{1} = \xi (12|34|5)\, , \;
   x^{\pa}_{2} = \xi (12|3|45)\, , \;
   x^{\pa}_{3} = \xi (1|23|45)\, , \;
   x^{\pa}_{4} = \xi (123|45)\, 
   \; \text{and} \;
   x^{\pa}_{5} = \xi (12|345)\ts ,
\]
where $\xi (\cA) := \varrho (\cA') / \bigl(\psi (\pmax|^{}_{\supp
  (\cA')}) - \psi(\cA')\bigr)$ with $\cA'$ being obtained from $\cA$
by removing all singleton parts, so $(12|34|5)' = (12|34)$ and so on.
In these expressions, the rates on subsystems are the induced rates
according to Eq.~\eqref{eq:rho-sub}. It is a somewhat surprising
feature that the final expressions take a relatively simple,
systematic form only after exploiting the recursive structure in this
way. \smallskip

If $x^{\pa}_{1} \cdot \ldots \cdot x^{\pa}_{5} \ne 0$, the function
$\theta$ is invertible, and its inverse $\eta$ is given by
\begin{footnotesize}
\[
\begin{array}{c|c @{\:\:} c @{\:\:} c @{\:\:} 
   c @{\:\:} c @{\;\:} c @{\;\,} c@{\;\,} c @{\;\:} c @{\;\:} 
   c @{\;\,} c @{\;\,} c @{\;\,} c @{\:\,} c @{\;\,} c @{\;\:} c}
  & \rotatebox[origin=c]{90}{$\,(1|2|3|4|5)\,$}
  & \rotatebox[origin=c]{90}{$(12|3|4|5)$}
  & \rotatebox[origin=c]{90}{$(1|23|4|5)$} 
  & \rotatebox[origin=c]{90}{$(1|2|34|5)$}
  & \rotatebox[origin=c]{90}{$(1|2|3|45)$}
  & \rotatebox[origin=c]{90}{$(123|4|5)$} 
  & \rotatebox[origin=c]{90}{$(12|34|5)$} 
  & \rotatebox[origin=c]{90}{$(12|3|45)$}
  & \rotatebox[origin=c]{90}{$(1|234|5)$}
  & \rotatebox[origin=c]{90}{$(1|23|45)$}
  & \rotatebox[origin=c]{90}{$(1|2|345)$} 
  & \rotatebox[origin=c]{90}{$(1234|5)$} 
  & \rotatebox[origin=c]{90}{$(123|45)$}
  & \rotatebox[origin=c]{90}{$(12|345)$}
  & \rotatebox[origin=c]{90}{$(1|2345)$} 
  & \rotatebox[origin=c]{90}{$(12345)$} \\
\hline
(1|2|3|4|5)^{\vphantom{\hat{I}}}
           & 1 & 1 & 1 & 1 & 1 & 1 & 1 & 1 & 1 & 1 & 1 & 1 & 1 & 1 & 1 & 1 \\
(12|3|4|5) & 0 & 1 & 0 & 0 & 0 & 1 & 1 & 1 & 0 & 0 & 0 & 1 & 1 & 1 & 0 & 1 \\
(1|23|4|5) & 0 & 0 & 1 & 0 & 0 & 1 & 0 & 0 & 1 & 1 & 0 & 1 & 1 & 0 & 1 & 1 \\
(1|2|34|5) & 0 & 0 & 0 & 1 & 0 & 0 & 1 & 0 & 1 & 0 & 1 & 1 & 0 & 1 & 1 & 1 \\
(1|2|3|45) & 0 & 0 & 0 & 0 & 1 & 0 & 0 & 1 & 0 & 1 & 1 & 0 & 1 & 1 & 1 & 1 \\
(123|4|5)  & 0 & 0 & 0 & 0 & 0 & 1 & 0 & 0 & 0 & 0 & 0 & 1 & 1 & 0 & 0 & 1 \\
(12|34|5)  & 0 & 0 & 0 & 0 & 0 & 0 & \frac{1}{x^{\pa}_{1}} & 0 & 0 & 0 
           & 0 & 1 & 0 & \frac{1}{x^{\pa}_{1}} & 0 & 1\\
(12|3|45)  & 0 & 0 & 0 & 0 & 0 & 0 & 0 & \frac{1}{x^{\pa}_{2}} & 0 & 0 
           & 0 & 0 & \frac{1}{x^{\pa}_{2}} & \frac{1}{x^{\pa}_{2}} & 0 & 1  \\
(1|234|5)  & 0 & 0 & 0 & 0 & 0 & 0 & 0 & 0 & 1 & 0 & 0 & 1 & 0 & 0 & 1 & 1 \\
(1|23|45)  & 0 & 0 & 0 & 0 & 0 & 0 & 0 & 0 & 0 & \frac{1}{x^{\pa}_{3}} 
           & 0 & 0 & \frac{1}{x^{\pa}_{3}} & 0 & 1 & 1 \\
(1|2|345)  & 0 & 0 & 0 & 0 & 0 & 0 & 0 & 0 & 0 & 0 & 1 & 0 & 0 & 1 & 1 & 1 \\
(1234|5)   & 0 & 0 & 0 & 0 & 0 & 0 & 0 & 0 & 0 & 0 & 0 & 1 & 0 & 0 & 0 & 1 \\
(123|45)   & 0 & 0 & 0 & 0 & 0 & 0 & 0 & 0 & 0 & 0 & 0 & 0 
           & \frac{1}{x^{\pa}_{4}} & 0 & 0 & 1 \\
(12|345)   & 0 & 0 & 0 & 0 & 0 & 0 & 0 & 0 & 0 & 0 & 0 & 0 
           & 0 & \frac{1}{x^{\pa}_{5}} & 0 & 1  \\
(1|2345)   & 0 & 0 & 0 & 0 & 0 & 0 & 0 & 0 & 0 & 0 & 0 & 0 & 0 & 0 & 1 & 1 \\
(12345)    & 0 & 0 & 0 & 0 & 0 & 0 & 0 & 0 & 0 & 0 & 0 & 0 & 0 & 0 & 0 & 1
\end{array}
\]
\end{footnotesize}

While the $\theta$-coefficients suffice to write down the solution in
the generic case, where we have $\psi^{U}\! (\pmax)\ne\psi^{U}\!
(\cA)$ for all $\cA\in\II (U)$ and all non-empty $U\nts\subseteq S$,
the degenerate cases require some care. First, for $\cA\in\{ \pmax,
(1|2345), (1234|5), (1|234|5)\}$, the solution is always given by
Eq.~\eqref{eq:simple-sol}. These cases correspond to the four rows in
the $\theta$-matrix of Table~\ref{tab:matrix} with constant entries,
which agree with the values of the M\"{o}bius function. For $\cA =
(12|345)$ and $\cA = (123|45)$, the general solution reads
\[
     a^{\pa}_{t} (\cA) \, = \, \varrho (\cA) \, 
     E^{\pa}_{0} (\psi (\pmax), \psi (\cA); t) \, = \,
     \ee^{-\psi (\pmax) \ts t} - \ee^{-\psi (\cA) \ts t}
     + \bigl(\psi (\cA) - \chi (\cA) \bigr)
     E^{\pa}_{0} (\psi (\pmax), \psi (\cA); t) 
\]
with the function $E^{\pa}_{0}$ from Eq.~\eqref{eq:E-fun}. Depending
on the values of the recombination rates, one can thus pick up a term
of the form $t \ee^{-\psi(\cA) \ts t}$, as in the case of four
sites. Next, one finds
\[
\begin{split}
    a^{\pa}_{t} (1|2|345) \, & = \, \ee^{-\psi (1|2|345) \ts t}
      - \ee^{-\psi(1|2345) \ts t} - \varrho(12|345) \, 
       E^{\pa}_{0} (\psi(\pmax), \psi (12|345) ; t) \ts , \\
    a^{\pa}_{t} (123|4|5) \, & = \, \ee^{-\psi (123|4|5) \ts t}
      - \ee^{-\psi(1234|5) \ts t} - \varrho(123|45) \, 
       E^{\pa}_{0} (\psi(\pmax), \psi (123|45) ; t) \ts ,
\end{split}    
\]
where the last term can be split as in the previous
case. Similarly, one gets
\[
\begin{split}
    a^{\pa}_{t} (12|3|45) \, = \, & \; \varrho^{\ts\prime} (12|45) \,
       E^{\pa}_{0} (\psi (\pmax), \psi (12|3|45); t )  \\
       & - \varrho (123|45) \,
       E^{\pa}_{0} (\psi (\pmax), \psi (123|45); t ) -
       \varrho (12|345) \,
       E^{\pa}_{0} (\psi (\pmax), \psi (12|345); t ) 
\end{split}
\]
with $\varrho^{\ts\prime} (12|45) = \varrho (12|3|45) + \varrho (123|45)
+ \varrho (12|345) $.  Here and below, the ${}^{\prime}$ indicates an
induced quantity on a subsystem (here with four sites).

All remaining partitions are true refinements of $(1|2345)$ or
$(1234|5)$, wherefore we may pick up a term of the form $t
\ee^{-\alpha t}$ already in the ODE, depending on the choice of
parameters. To deal with this complication, let us first state a
simple variant of the results in \cite[Appendix]{main}.

\begin{fact}\label{fact:ODE}
  Let\/ $\rho$ as well as\/ $\sigma^{\pa}_{1}, \dots ,\sigma^{\pa}_{m}$ 
  and\/ $\sigma^{\ts\prime}_{1},\dots ,\sigma^{\ts\prime}_{n}$ be arbitrary
  non-negative numbers, and\/ $\varphi$ a continuous function
  on\/ $\RR^{\pa}_{\geqslant 0}$ such that\/ $\ee^{\rho\ts t} \ts
  \varphi (t)$ is integrable. Then, the Cauchy problem defined by the
  ODE
\[
     \dot{g} \, = \, - \rho\, g \, + \, \varphi (t) \, +
     \sum_{i=1}^{m} \varepsilon^{\pa}_{i} \,
     (\rho - \sigma^{\pa}_{\nts i})\, \ee^{-\sigma^{\pa}_{i} t} \, +
     \sum_{j=1}^{n} \varepsilon^{\ts\prime}_{\nts j} \,
     \ee^{-\sigma^{\ts\prime}_{\! j}\ts t}
\]  
together with the initial condition\/ $g (0) = g^{\pa}_{0}$ and\/
$\varepsilon^{\pa}_{i} , \varepsilon^{\ts\prime}_{\nts j} \in \RR$ has the
unique solution
\[
      g \, = \, g^{\pa}_{0}\, \ee^{-\rho \ts\ts t} +
      \ee^{-\rho \ts\ts t} \int_{0}^{t} \ee^{\ts \rho \ts s} \ts
      \varphi (s) \dd s \,  + \sum_{i=1}^{m} \varepsilon^{\pa}_{i}
      \bigl( \ee^{-\sigma^{\pa}_{\nts i}  t} - \ee^{-\rho \ts\ts t}\bigr) +
      \sum_{j=1}^{n} \varepsilon^{\ts\prime}_{\nts j} \,
      E^{\pa}_{0} (\rho, \sigma^{\ts\prime}_{\! j}; t) \ts ,
\]  
  which holds for all\/ $t\geqslant 0$.   \qed
\end{fact}

Let us now consider the case $\cA = (12|34|5)$ in more detail.
According to Eq.~\eqref{eq:a-ODE}, the ODE we have to solve here,
after some calculations, reads
\[
\begin{split}
    \dot{a}^{\pa}_{t} (\cA) \, = \, & - \, 
     \psi (\pmax) \, a^{\pa}_{t} (\cA)
     \; + \!\! \sum_{\cA \preccurlyeq \udo{\cB} \prec \pmax} \!
    \mu(\cA,\cB) \bigl( \psi (\pmax) - \psi (\cB)\bigr)
     \ee^{-\psi (\cB) \ts t}  \\
     & + \, c \,  E^{\pa}_{0} (\psi (1234|5), \psi (\cA); t) \, +
     \!\! \sum_{\cA \preccurlyeq \udo{\cB} \prec \pmax} \!
    \mu(\cA,\cB) \bigl( \psi (\cB) - \chi (\cB)\bigr)
     \ee^{-\psi (\cB) \ts t}
\end{split}     
\]
where $c=\varrho (1234|5) \bigl(\varrho^{\ts\prime} (1|2|3|4) +
\varrho^{\ts\prime} (1|23|4) \bigr)$.  This is of the type covered by
Fact~\ref{fact:ODE}, with $\varphi (t) = c \, E^{\pa}_{0} (\psi
(1234|5), \psi (\cA); t)$. Note that the two rates here generically
differ from $\rho = \psi (\pmax)$.  Depending on the choice of the
recombination rates, the solution can be a sum of exponentials (with
coefficients given by $\theta$), which is the generic case.  However,
it can also contain a contribution of the form $t \ts \ee^{-\psi (\cB)
  \ts t}$ (when $\chi (\cB) \ne \psi (\cB) = \psi (\pmax)$ occurs) or
even of the form $t^{2}\ts \ee^{-\psi (\cA) \ts t}$, which happens for
the double degeneracy that $\psi (\cA) = \psi (1234|5)$ together with
$\psi(\cA)= \psi (\pmax)$, as a consequence of \cite[Lemma~9]{main}.
We leave it to the reader to analyse the details of the various cases
that are possible here.  Clearly, the treatment of the partition
$(1|23|45)$ is completely analogous, so that we have covered all
partitions with up to three parts.

The remaining five cases can be analysed in the same way. Doing so,
one realises that the general form of the ODE for $a^{\pa}_{t} (\cA)$
with arbitrary $\cA \in \II (S)$ reads
\begin{equation}\label{eq:ODE-5}
\begin{split}
  \dot{a}^{\pa}_{t} (\cA) \, = \, & - \psi (\pmax)\, a^{\pa}_{t} (\cA)
   + \! \sum_{\cA \preccurlyeq\udo{\cB}\prec \pmax} \!
   \mu (\cA, \cB) \bigl[ \bigl( \psi (\pmax) - \psi (\cB)\bigr)
   + \bigl(\psi (\cB) - \chi (\cB) \bigr)\bigr]
     \ee^{-\psi (\cB) \ts t} \\
  & + \mu \bigl(\cA, (12|34|5) \bigr)\, c \,
      E^{\pa}_{0} (\psi(1234|5), \psi(12|34|5); t) \\[1mm]
  &  + \mu \bigl(\cA, (1|23|45) \bigr)\, \tilde{c} \,
      E^{\pa}_{0} (\psi(1|2345), \psi(1|23|45); t) \ts ,
\end{split}
\end{equation}
with $c = \varrho (1234|5) \bigl(\varrho^{\ts\prime} (1|2|3|4) +
\varrho^{\ts\prime} (1|23|4)\bigr)$ and $\tilde{c} = \varrho (1|2345)
\bigl(\varrho^{\ts\prime} (2|3|4|5) + \varrho^{\ts\prime}
(2|34|5)\bigr)$. The solution is fully covered by Fact~\ref{fact:ODE},
and the necessary case distinctions should be clear. In particular,
the solutions will contain simple exponentials, contributions of type
$E^{\pa}_{0}$ (whenever $\psi (\cB) \ne \chi (\cB)$ in the sum, which
leads to terms of the form $t \ts \ee^{-\alpha\ts t}$), but possibly
also integrals as in the previous two cases (with the same
consequences, namely possible terms of the form $t^{2} \ee^{-\alpha\ts
  t}$ in the case of double degeneracies). Let us summarise this as
follows.

\begin{coro}\label{coro:sol-5-sites}
  The solution of the general recombination equation
  \eqref{eq:main-ODE} on\/ $\cM (X)$, for the set\/ $S=\{ 1,2,3,4,5
  \}$ with the lattice\/ $\VV = \II (S)$ and with a probability
  measure\/ $\omega^{\pa}_{0} \in \cP (X)$ as initial condition, is
  given by
\[
      \omega^{}_{t} \, =
      \sum_{\cA\in\II (S)} a^{\pa}_{t} (\cA) \, 
      R^{\pa}_{\nts\cA} (\omega^{\pa}_{0}) \ts ,
\]      
where the convex coefficients\/ $a^{\pa}_{t} (\cA)$ are the unique
solution of the Cauchy problem defined by Eq.~\eqref{eq:ODE-5}
together with the initial condition\/ $a^{\pa}_{0} (\cA) = \delta
(\cA, \pmax)$, as given in Fact~$\ref{fact:ODE}$.  \qed
\end{coro}

Beyond five sites, the complexity of the solution in the presence of
degeneracies increases. In particular, for $n$ sites, one can in
principle also obtain terms of the form $t^{m} \ee^{-\alpha t}$ for
any $0 \leqslant m \leqslant n-3$, though this requires an $m$-fold
degeneracy in analogy to the appearance of double degeneracies in our
above examples. Since the generic case without any of the `bad'
degeneracies most likely covers the practically relevant cases, we do
not want to expand on this issue any further. We shall come back to
it from a different point of view shortly.

\section{Further directions}\label{sec:outlook}

Let us return to the general case of $n$ sites, with an irreducible
and indecomposable partition lattice $\VV$.  At this stage, we know
that $\omega^{\pa}_{t} = \sum_{\cA \in \VV} a^{\pa}_{t} (\cA) \,
R^{\pa}_{\!\cA} (\omega^{\pa}_{0})$ solves the recombination ODE
\eqref{eq:main-ODE} when each $a^{\pa}_{t} (\cA)$ solves
Eq.~\eqref{eq:a-ODE}, both with the properly corresponding initial
conditions. In the generic case with $a^{\pa}_{0} (\cA) = \delta
(\cA,\pmax)$, this means $a^{\pa}_{t} (\cA) = \sum_{\cB\succcurlyeq
  \cA} \theta (\cA,\cB)\ts \ee^{-\psi (\cB)\ts t}$, where one also has
$\ee^{-\psi (\cB)\ts t} = \sum_{\cC\succcurlyeq \cB} \eta (\cB,\cC)\,
a^{\pa}_{t} (\cC)$.  This structure thus provides a \emph{mode
  decomposition} (also known as decomposition into linkage
disequilibria) $\omega^{\pa}_{t} = \sum_{\cA\in\VV} \nu^{\pa}_{t}
(\cA)$ with
\[
\begin{split}
   \nu^{\pa}_{t} (\cA) \, & = \,  \ee^{-\psi (\cA)\ts t}
   \sum_{\cC \preccurlyeq \cA} \theta (\cC,\cA) \,
   R^{\pa}_{\cC} (\omega^{\pa}_{0} ) \\
   & = \sum_{\cB\succcurlyeq \cA} \eta (\cA,\cB) \,
   a^{\pa}_{t} (\cB)  \sum_{\cC \preccurlyeq \cA} \theta (\cC,\cA) \,
   R^{\pa}_{\cC} (\omega^{\pa}_{0} ) \ts ,
\end{split}
\]
where the (signed) measure $\nu^{\pa}_{t} (\cA)$ satisfies the
linear ODE
\[
     \dot{\nu}^{\pa}_{t} (\cA) \, = \,
     - \psi (\cA)\,  \nu^{\pa}_{t} (\cA)
\]
together with the initial condition $\nu^{\pa}_{0} (\cA) =
\sum_{\cC\preccurlyeq \cA} \theta (\cC,\cA) \, R^{\pa}_{\cC}
(\omega^{\pa}_{0} )$.

Since our general assumptions on the recombination rates (for $\VV$
being indecomposable and irreducible) imply that $\psi (\cA) >0$ for
all $\cA \ne \pmin$, we get the norm convergence
\[
    \omega^{\pa}_{t} \, \xrightarrow{\, t \to \infty \,} \,
    R^{\pa}_{\pmin} (\omega^{\pa}_{0} ) \ts ,
\]
with exponential decay of all modes $\nu^{\pa}_{0} (\cA)$ for $\cA \ne
\pmin$. For any $\omega^{\pa}_{0} \in \cP (X)$, there is thus the
\emph{unique} equilibrium, $R^{\pa}_{\pmin} (\omega^{\pa}_{0} )$, and
the convergence towards it is exponentially fast. In the degenerate
cases, as explained above, this situation may have to be modified to
include mode decay of the form $t^{m} \ee^{-\alpha\ts t}$ for some
$m$, where $m \leqslant n-3$ with $n$ the number of sites, but this
does not change the equilibrium. As the general structure should be
clear, we leave further details to the reader.

Let us return to the generic case that $\psi^{U}\! (\pmax) \ne
\psi^{U}\! (\cA)$ holds for all $\cA\in \VV (U)$ and all non-empty
$U\nts\subseteq S$. Here, we are in the regime of the recursive
formula \eqref{eq:rec} for $\theta^{S}$.  Though it seems difficult to
write down a closed formula for $\theta^{S}\!$, we can say more on the
diagonal elements. Define $\vartheta^{U}\! (\cA) = \theta^{U}\!
(\cA,\cA)$ and set $\vartheta^{\varnothing} (\varnothing)=1$.  This
way, $\vartheta$ is an element of the M\"{o}bius algebra, and
satisfies the recursion
\begin{equation}\label{eq:theta-rec}
    \vartheta^{U}\! (\cA) \, = \! \sum_{\cA\preccurlyeq \udo{\cB} \prec \pmax}
    \! \varrho^{U}\! (\cB) \, \prod_{i=1}^{\lvert \cB \rvert}
    \vartheta^{B_{i}} (\cA |^{\pa}_{B_{i}} )\ts ,
\end{equation}
which is an immediate consequence of Eq.~\eqref{eq:rec} and holds under
the non-degeneracy conditions mentioned above. As introduced earlier,
let $\cA^{\prime}$ be the partition that emerges from $\cA$ by
removing all singleton parts, which includes the limiting case
$\pmin^{\ts\prime} = \varnothing$.

\begin{prop}\label{prop:diag-reduction}
  Let\/ $S$ be a finite set and\/ $\VV$ an indecomposable and
  irreducible lattice of partitions of\/ $S$. For the generic choice
  of the recombination rates\/ $\varrho^{S}$, the recursively defined
  functions\/ $\vartheta^{U}\!\in\MM^{\pa}_{\RR} (\VV (U))$ introduced
  above satisfy the relations
\[
      \vartheta^{U}\! (\cA) \, = \, 
      \vartheta^{\supp (\widetilde{\cA}\, )} \bigl(\widetilde{\cA}\, \bigr)
\]    
for all\/ $\cA\in\VV (U)$ and all non-empty\/ $U\nts\subseteq S$.
Here, $\widetilde{A}$ is any partition that emerges from\/ $\cA$ by
the removal of an arbitrary number of singletons, including the case\/
$\widetilde{\cA}=\cA^{\prime}$.
\end{prop}

\begin{proof}
  The claim is true for all $\cA\in\PP(S)$ with any $S$ of cardinality
  $\lvert S \rvert \leqslant 3$ as a consequence of $\vartheta^{S}
  (\cA) = \mu^{S} (\cA,\cA) = 1$ for all $\cA\in\PP(S)$ in this case,
  hence also for the sublattice $\VV$. The formula is also true for
  arbitrary $U \ne \varnothing$ when $\cA$ contains no singletons or when
  $\widetilde{\cA}=\cA$ (both sides are equal) or when $\cA = \pmin$
  together with $\widetilde{\cA}=\cA^{\prime}=\varnothing$ (where both
  sides evaluate to $1$).

  To continue the proof by induction in the number of elements of $S$
  respectively $U$, we may thus assume that $\cA = \cD \sqcup \cE$ is
  the join of a partition $\cD$, which may or may not contain one or
  more singletons, with another (non-empty) partition $\cE$ that
  \emph{only} contains singletons. With $D:=\supp (\cD)\ne
  \varnothing$, we can then calculate
\[
    \bigl( \psi^{S} (\pmax) -  \psi^{S} (\cA)\bigr)
    \ts\ts \vartheta^{S} (\cA) \,
     =  \!\sum_{\cA\preccurlyeq\udo{\cB}\prec\pmax}\! \varrho^{S} (\cB)
    \, \prod_{i=1}^{\lvert \cB \rvert} \vartheta^{B_{i}}
     (\cA |^{\pa}_{B_{i}}) \,
     = \, \sum_{\cF\succcurlyeq\cD} \, \sum_{\substack{\cC\ne\pmax \\
         \cC |^{\pa}_{D} = \cF}} \!\varrho^{S} (\cC)\,
         \prod_{i=1}^{\lvert \cC \rvert} \vartheta^{C_{i}}
     (\cA |^{\pa}_{C_{i}}) \ts .
\]
Using the induction hypothesis and the structure of $\cA$, the last
product can be simplified as
\[
    \prod_{i=1}^{\lvert \cC \rvert} \vartheta^{C_{i}} (\cA |^{\pa}_{C_{i}})
    \, = \, \prod_{j=1}^{\lvert \cF \rvert} \vartheta^{F_{j}}
    (\cD |^{\pa}_{F_{j}}) \ts ,
\]
which allows to evaluate the inner sum over $\cC$ as $\varrho^{D}
(\cF) - \varrho^{S} (\pmax) \, \delta (\cF, \pmax |^{\pa}_{D} )$.
This gives
\[
\begin{split}
    \bigl( \psi^{S} (\pmax) -  \psi^{S} (\cA)\bigr) 
          \ts \vartheta^{S} (\cA) \,
     & =  \ts\bigl( \varrho^{D} (\pmax |^{\pa}_{D})
           - \varrho^{S} (\pmax) \bigr)\ts
     \vartheta^{D} (\cD) \, + \! \sum_{\cD\preccurlyeq\udo{\cF}\prec\pmax}
      \! \varrho^{D} (\cF) \prod_{j=1}^{\lvert \cF \rvert} 
     \vartheta^{F_{j}} (\cD |^{\pa}_{F_{j}}) \\
     & = \ts\bigl( \psi^{D} (\pmax |^{\pa}_{D} ) 
       + \varrho^{D} (\pmax |^{\pa}_{D}) - \varrho^{S} (\pmax)
       - \psi^{D} (\cD) 
       \bigr) \ts\ts \vartheta^{D} (\cD) \\[3mm]
     & = \ts\bigl(  \psi^{S} (\pmax) -  \psi^{S} (\cA)\bigr) \ts\ts 
     \vartheta^{S} (\cD) \ts , 
\end{split}     
\]
where we used the recursion \eqref{eq:theta-rec} as well as the
identity $\psi^{D} (\cD) = \psi^{S} (\cA)$ and the relation between
$\psi^{S} (\pmax)$ and $\psi^{D} (\pmax |^{\pa}_{D})$, which both
follow from the definition of the decay rates. Since $\psi^{S} (\pmax)
\ne \psi^{S} (\cA)$ by assumption, our claim follows.
\end{proof}

Note that $\vartheta$ coincides with the function $\xi$ used after
Corollary~\ref{coro:all-clear-five} to calculate the recombination
rate dependent terms in Table~\ref{tab:matrix}. Inspecting this table
again together with the recursion for $\theta$ from
Eq.~\eqref{eq:rec}, one sees that a similar reduction as for
$\vartheta$ must be valid for $\theta$, too. Indeed, if $\cA
\preccurlyeq \cB$, one can replace the pair $(\cA,\cB)$ by a reduced
one that emerges via the removal of all \emph{common} singletons in
$\cA$ and $\cB$. With $\theta^{\varnothing} (\varnothing,\varnothing)
:= 1$, the reasoning of the proof of
Proposition~\ref{prop:diag-reduction}, applied to the recursion
\eqref{eq:rec} for $\theta^{U}\!$, gives the following result.

\begin{theorem}\label{thm:reduction}
  Let\/ $S$ and\/ $\VV$ be as in
  Proposition~$\ref{prop:diag-reduction}$. For the generic choice of
  the recombination rates\/ $\varrho^{S}$, the functions\/
  $\theta^{U}\!\in\AAA^{\pa}_{\RR} (\VV (U))$ satisfy the relations
\[
      \theta^{U}\! (\cA,\cB) \, = \,
      \theta^{\supp (\widetilde{\cB}\, )} 
      \bigl(\cA |^{\pa}_{\supp (\widetilde{\cB}\, )},\widetilde{\cB}\, \bigr)
\]    
for all\/ $\cA,\cB\in\VV (U)$ with\/ $\cA\preccurlyeq \cB$ and all\/
$\varnothing\ne U\nts\subseteq S$. Here, $\widetilde{B}$ is any
partition that emerges from\/ $\cB$ by the removal of an arbitrary
number of singletons, including the case\/
$\widetilde{\cB}=\cB^{\prime}$. \qed
\end{theorem}

This reduction property helps to better understand the values of the
$\theta$-coefficients and their recursive structure. Moreover, it also
explains why and how the entries of the matrix in
Eq.~\eqref{eq:theta-4-sites} reappear within Table~\ref{tab:matrix}
for five sites, once for partitions that comprise the singleton $\{ 5
\}$ and once for those comprising $\{ 1 \}$.

Invoking Lemma~\ref{lem:inverse} from the Appendix and observing that,
for $\cA \preccurlyeq \cB$, $\chn (\cA,\cB)$ and $\chn
(\cA|^{\pa}_{\supp (\cB^{\prime})}, \cB^{\prime})$ have the same
cardinality, it is clear that the reduction property for $\theta$
extends to its inverse $\eta$ as follows.
\begin{coro}\label{coro:inv-reduction}
  Let the assumption be as in Proposition~$\ref{prop:diag-reduction}$
  and Theorem~$\ref{thm:reduction}$, and assume that the recombination
  rates are chosen such that\/ $\theta^{S} \in \AAA^{\pa}_{\RR} (\VV)$
  is well-defined and invertible, with inverse\/ $\eta^{S}$. Then, one
  also has
\[
      \eta^{S} (\cA,\cB) \, = \,
      \eta^{\supp (\widetilde{\cB}\, )} 
      \bigl(\cA |^{\pa}_{\supp (\widetilde{\cB}\, )},\widetilde{\cB}\, \bigr)
\]    
for all\/ $\cA,\cB\in\VV$ with\/ $\cA\preccurlyeq \cB$. As before,
$\widetilde{B}$ is any partition that emerges from\/ $\cB$ by the
removal of an arbitrary number of singletons, including the case\/
$\widetilde{\cB}=\cB^{\prime}$. \qed
\end{coro}

Let us finally take a closer look at the structure of the probability
vectors $a^{\pa}_{t}$ as a function of time, first in the simpler case
that $\psi=\chi$, with $\rtot{} := \sum_{\cA\in\VV} \varrho (\cA)$.
\begin{prop}\label{prop:lin-Markov}
  Let\/ $S$ be a finite set and consider the lattice\/ $\II (S)$.
  Moreover, assume that the recombination rates\/ $\varrho (\cA)$ are
  chosen so that\/ $\psi = \chi$ on\/ $\II (S)$. Then, the solution\/
  $a^{\pa}_{t}$ of Eq.~\eqref{eq:a-ODE} with initial condition\/
  $a^{\pa}_{0} (\cA) = \delta (\cA,\pmax)$ also solves the linear ODE
  system
 \[
 \begin{split}
       \dot{a}^{\pa}_{t} (\cA) \, & = \, - \sum_{\cB \succcurlyeq \cA}
       Q (\cA, \cB) \, a^{\pa}_{t} (\cB)
       \qquad \text{with} \\[1mm]
       Q (\cA,\cB) \, & = \, - \!\sum_{\cC \in [\cA,\cB]}\!
       \mu (\cA,\cC) \, \chi (\cC) \, = \,
       - \ts \rtot{} \, \delta (\cA,\cB) \, + \sum_{\cC \succcurlyeq \cA}
       \varrho (\cC) \, \delta (\cA, \cB\nts\wedge\cC) \ts , 
 \end{split}      
 \]   
     where\/ $Q$ is $($the transpose of$\, )$ a Markov generator,
     with the diagonal entries\/ $Q (\cA,\cA) = - \chi (\cA)$
     being the eigenvalues of\/ $Q$.
 \end{prop}    

\begin{proof}
  Under the assumptions, we know $a^{\pa}_{t} (\cA) =
  \sum_{\cC\succcurlyeq\cA} \mu(\cA,\cC) \, \ee^{-\chi (\cC)\ts t}$
  from Proposition~\ref{prop:reduction-to-linear} together with
  $\ee^{-\chi (\cC)\ts t} = \sum_{\cB\succcurlyeq\cC} a^{\pa}_{t}
  (\cB)$ by M\"{o}bius inversion. Consequently, one has
\[
\begin{split}
     \dot{a}^{\pa}_{t} (\cA) \, & = \, 
     - \sum_{\cC\succcurlyeq\cA} \mu(\cA,\cC)\,
     \chi(\cC)\, \ee^{-\chi (\cC)\ts t} \, = \,
     - \sum_{\cC\succcurlyeq \cA} \mu(\cA,\cC) \, \chi(\cC)
     \sum_{\cB\succcurlyeq\cC} a^{\pa}_{t} (\cB)   \\[1mm]
     & = \, \sum_{\cB\succcurlyeq\cA} \biggl( - \!
     \sum_{\cC\in[\cA,\cB]}\! \mu(\cA,\cC) \, \chi(\cC) \biggr)
     \, a^{\pa}_{t} (\cB) \ts ,
\end{split}
\]
which proves the first claim and the first formula for $Q$.

Clearly, $Q (\cA,\cB) = 0$ whenever $\cA \not\preccurlyeq \cB$. If
$\cA \preccurlyeq \cB$, one uses Eq.~\eqref{eq:chi-sum} to get
\[
\begin{split}
   Q (\cA,\cB) \, & = \, - \! \sum_{\cC\in[\cA,\cB]} \! \mu (\cA,\cC) \,
   \chi(\cC) \,   = \, - \! \sum_{\cC\in[\cA,\cB]} \! \mu (\cA,\cC)
   \biggl( \rtot{} - \! \sum_{\cD\in[\cC,\pmax]} \varrho (\cD) \biggr) \\[1mm]
   & = \, - \ts \rtot{} \, \delta (\cA,\cB) \, + \sum_{\cD \succcurlyeq \cA}
   \varrho (\cD) \! \sum_{\cC \in [\cA,\cD] \cap [\cA,\cB]} \! \mu (\cA,\cC) \ts .
\end{split}   
\]
Since $[\cA,\cB]\cap[\cA,\cD] = [\cA,\cB\nts\wedge\cD]$, the last sum
is $\sum_{\cC\in[\cA,\cB\nts\wedge\cD]}\mu(\cA,\cC) = \delta
(\cA,\cB\nts\wedge\cD)$ and the second formula for $Q$ follows.

All off-diagonal terms of $Q$ are then clearly non-negative, while
\[
    \sum_{\cA\in\II (S)} \!\! Q (\cA,\cB) \, = \, - \ts \rtot{} +
    \sum_{\cA\preccurlyeq\cB} \, \sum_{\cC\succcurlyeq\cA}
    \varrho (\cC)\, \delta(\cA, \cB\nts\wedge\cC) \, = \,
    -\ts\rtot{} + \!\! \sum_{\cC\in\II (S)} \!\! \varrho (\cC) \!
    \! \sum_{\cA\preccurlyeq\cB\nts\wedge\cC} 
    \!\! \delta (\cA, \cB\nts\wedge\cC) \, = \, 0
\]
because $\cB\nts\wedge\cC$ is a unique partition in $\II (S)$, so that
the last sum is always $1$. This shows the claimed Markov generator
property. The diagonal entries are clear from the first formula for
$Q$. They are the eigenvalues because $Q$ is upper triangular.
\end{proof}

Note that the second expression for $Q$ in
Proposition~\ref{prop:lin-Markov} is slightly misleading in the sense
that it suggests that all rates $\varrho (\cC)$ with $\cB\nts\wedge\cC
= \cA$ contribute to $Q (\cA,\cB)$. However, the condition $\psi =
\chi$ forces $\varrho$ to vanish for many partitions as $\lvert S
\rvert$ increases. In fact, as we shall see from
Eq.~\eqref{eq:Q-formula} below, $Q (\cA,\cB) > 0$ is only possible if
$\cA \prec \cB$ and if $\cA$ refines precisely one part of $\cB$. This
is a somewhat surprising consequence of the underlying stochastic
partitioning process discovered in \cite[Sec.~6]{main}. To explore
this connection, one can now apply an analogous reasoning as in
Proposition~\ref{prop:lin-Markov} to the general generic case, with
the following result.
\begin{theorem}\label{thm:gen-Markov}
  Let\/ $\VV$ be an indecomposable and irreducible lattice of
  partitions of a finite set and assume that the recombination rates\/
  $\varrho (\cA)$ are such that\/ $\theta \in \AAA^{\pa}_{\RR} (\VV)$
  is well-defined and invertible, with inverse\/ $\eta$.  Then, the
  solution\/ $a^{\pa}_{t}$ of the ODE system~\eqref{eq:a-ODE} with
  initial condition\/ $a^{\pa}_{0} (\cA) = \delta (\cA,\pmax)$ also
  solves the\/ \emph{linear} ODE system
 \[
       \dot{a}^{\pa}_{t} (\cA) \,  = \, - \sum_{\cB \succcurlyeq \cA}
       Q (\cA, \cB) \, a^{\pa}_{t} (\cB)
       \quad \text{with} \quad
       Q (\cA,\cB) \,  = \, - \!\sum_{\cC \in [\cA,\cB]}\!
       \theta (\cA,\cC) \, \psi (\cC) \, \eta (\cC,\cB) \ts ,
 \] 
 where\/ $Q$ is $($the transpose of$\, )$ a Markov generator.
 Moreover, $Q$ is upper triangular with diagonal entries\/ $Q
 (\cA,\cA) = -\psi (\cA)$, which are the eigenvalues of\/ $Q$.
 \end{theorem}    
 \begin{proof}
   The first claim follows from Eq.~\eqref{eq:theta-sum} in
   conjunction with Theorem~\ref{thm:theta-rec}. The calculation runs
   as in the proof of Proposition~\ref{prop:lin-Markov}, now with
   $\ee^{-\psi(\cC)\ts t} = \sum_{\cB\succcurlyeq\cC} \eta (\cC,\cB)\,
   a^{\pa}_{t} (\cB)$ via the inversion formula for
   Eq.~\eqref{eq:theta-sum}. Clearly, we have $Q (\cA,\cB) = 0$
   whenever $\cA \not\preccurlyeq\cB$, which implies that $Q$ is upper
   triangular. The diagonal entries are then the eigenvalues, and
   their values are a consequence of $\eta (\cA,\cA) = 1/\theta
   (\cA,\cA)$.

Next, for fixed $\cB \in \VV$, one finds the column sums
 \[
    \sum_{\cA\in\VV} Q (\cA,\cB) \, = \, -  \sum_{\cA\preccurlyeq\cB}
    \, \sum_{\cC\in[\cA,\cB]} \theta (\cA,\cC) \, \psi (\cC)\,
    \eta (\cC,\cB) \, = \sum_{\cC\preccurlyeq\cB} \psi (\cC) \,
    \eta (\cC,\cB) \sum_{\cA\preccurlyeq\cC} \theta (\cA,\cC) \ts .
 \]
 Since the last sum is $\delta (\pmin, \cC)$ by \cite[Prop.~8
 (2)]{main}, which holds in this generality, one has
 \[
    \sum_{\cA\in\VV} Q (\cA,\cB) \, = \, \psi (\pmin) \, \eta (\pmin, \cB)
    \,  = \, 0
 \]
 from $\psi (\pmin) = 0$, as is necessary for (the transpose of) a
 Markov generator. Since $Q$ is upper triangular, it remains to show
 that $Q (\cA, \cB) \geqslant 0$ whenever $\cA \prec \cB$.

To this end, we employ the underlying partitioning process from
\cite[Sec.~6]{main}. Our coefficients $a^{\pa}_{t} (\cA)$ were
identified as the first row of a Markov semigroup, namely as
transition probabilities $\bs{P}^{\pa}_{t} (\pmax \to \cA)$. The
corresponding Markov generator $\widetilde{Q}$ could be identified by
also considering $\bs{P}^{\pa}_{t} (\cB \to \cA)$ for an arbitrary
$\cB \in \VV$, leading to the so-called Kolmogorov \emph{backward}
equation. By an application of \cite[Thm.~2.1.1]{N}, we can switch to
the corresponding \emph{forward} equation with the same generator,
thus giving a linear ODE system for the quantities $\bs{P}^{\pa}_{t}
(\pmax \to \cA)$ and hence $a^{\pa}_{t} (\cA)$ \emph{alone}, which is
nothing but the equation stated in the theorem, with $Q =
\widetilde{Q}^{T}$. Consequently, all off-diagonal elements of $Q$ are
non-negative.
\end{proof}

 Since it is difficult to calculate the off-diagonal elements of
 $Q$ from the formula in Theorem~\ref{thm:gen-Markov}, we employ
 \cite[Thm.~3]{main} to determine them. As mentioned earlier, $Q (\cA,
 \cB)>0$ requires $\cA \prec \cB$ with $\cA$ refining precisely one
 part of $\cB$.  So, we may assume this part to be $B^{\pa}_{1}$,
 hence $\cB = \{ B^{\pa}_{1}, B^{\pa}_{2}, \dots , B^{\pa}_{r}\}$ with
 $r = \lvert \cB\rvert$ and $\cA = \{ A^{\pa}_{1}, \dots ,A^{\pa}_{s},
 B^{\pa}_{2}, \dots , B^{\pa}_{r}\}$, with $s\geqslant 2$ and obvious
 meaning for $r=1$, where $\lvert\cA\rvert = s+r-1$ and
 $\cA|^{\pa}_{B^{\pa}_{1}} = \{ A^{\pa}_{1}, \dots ,A^{\pa}_{s}\}$.
 Then, we get
\begin{equation}\label{eq:Q-formula}
    Q (\cA, \cB) \, = \, \varrho^{B^{\pa}_{1}} (\cA|^{\pa}_{B^{\pa}_{1}})
    \, =\!\!\! \sum_{\substack{\cC\in\VV\; \\
     \; \cC|^{\pa}_{B^{\pa}_{1}}\! =\, \cA|^{\pa}_{B^{\pa}_{1}} }} 
    \!\!\! \varrho (\cC) \ts ,
\end{equation}
which clearly is always non-negative. One can now calculate that
$\sum_{\cA\prec\cB} Q (\cA,\cB) = \psi (\cB)$, in line with $Q$ having
vanishing column sums.  It is also instructive to calculate the matrix
$Q$ for the lattice $\VV = \II (\{1,2,3,4\})$ in both ways (via the
formula from Theorem~\ref{thm:gen-Markov} and via
Eq.~\eqref{eq:Q-formula}) to get an impression of the miraculous
cancellations that happen along the way. The general verification of
the equality of the two formulas is more difficult, and can partly be
based on the reduction properties for $\theta$ and $\eta$, but also
needs some lengthy calculations that we omit.

Now, the Markov generator $Q$ has a surprisingly simple form, and
certainly no singularity of any kind. In fact, one can calculate it
for $\VV = \PP (S)$, where it holds for \emph{all} values of the
recombination rates $\varrho$, and then restrict it to any sublattice,
even to decomposable or reducible ones. The role of the somewhat
tricky degeneracies appear in a new light then, too: Depending on the
values of the parameters, the matrix $Q$ will or will not be
diagonalisable. In the former case, we are in the regime where the
recursive solution with the ansatz from Eq.~\eqref{eq:theta-sum}
works. In the latter case, we have to employ the well-known theory of
linear ODE systems for Jordan blocks, compare \cite[Thm.~12.7]{Amann}
or \cite[Ch.~17.VII]{Wal}, thus giving the appearance of terms of the
form $t^m \ee^{-\alpha\ts t}$ another and more transparent meaning.
Further ramifications of this connection will be explored in a
separate publication \cite{new}.

\clearpage

\section*{Appendix: An inversion lemma for the incidence 
algebra}

Let $L$ be a \emph{poset} with ordering relation $\preccurlyeq$. We
assume $L$ to be \emph{locally finite}, which means that all intervals
$[\cA,\cB]$ in $L$ are finite; compare \cite[Chs.~II and IV]{Aigner}
for background. Now, let $\AAA^{\pa}_{K} (L)$ be the corresponding
\emph{incidence algebra} (over any field $K$ of characteristic $0$).
This is an associative $K$-algebra with the Kronecker function
$\delta$ as its (two-sided) identity. Any element of the algebra,
$\alpha$ say, satisfies $\alpha (\cA,\cB) = 0$ whenever $\cA,\cB \in
L$ with $\cA \not\preccurlyeq \cB$, and the multiplication of two
elements is defined by
\[
      (\alpha * \beta) (\cA,\cB) \; = 
      \sum_{\cC \in [\cA,\cB]} \alpha (\cA,\cC) \, \beta(\cC,\cB) \ts .
\]
An element $\alpha\in\AAA (L)$ is \emph{invertible} if and only if
$\alpha (\cA,\cA) \ne 0$ for all $\cA\in L$. In this case, its unique
inverse $\alpha^{-1}$ (which is both the left and the right inverse,
so $\alpha^{-1} * \alpha = \alpha * \alpha^{-1} = \delta$) is
recursively given by $\alpha^{-1} (\cA,\cA) = 1/\alpha (\cA,\cA)$
together with
\begin{equation}\label{eq:rec-inverse}
     \alpha^{-1} (\cA,\cB) \, = \,
     \frac{-1}{\alpha (\cA,\cA)} \sum_{\cA \prec \udo{\cD} \preccurlyeq \cB}
     \alpha (\cA,\cD) \, \alpha^{-1} (\cD,\cB) \ts .
\end{equation}

It is sometimes useful to have a closed expression for $\alpha^{-1}$
that is not recursive in nature. One classic identity, see
\cite[Thm.~3.3.22]{SD}, uses a summation over chains. If $\cA
\preccurlyeq \cB$, a \emph{chain} from $\cA$ to $\cB$ is any
increasing sequence $G \subseteq L$ of the form
\[
     \cA = \cG_1 \prec \cG_2 \prec \dots \prec \cG_{\ell} = \cB \ts ,
\]
where $\lvert G \rvert = \ell$ is called the \emph{length} of $G$
(note that some authors use $\lvert G \rvert - 1$ for the length). We
use $\chn (\cA,\cB)$ to denote the set of all chains from $\cA$ to
$\cB$.  When $\cA \not \preccurlyeq \cB$, $\chn (\cA,\cB) =
\varnothing$ as usual.

\begin{lemma}\label{lem:inverse}
   If\/ $\alpha \in \AAA (L)$ is invertible, its inverse is given by
\[
    \alpha^{-1} (\cA,\cB) \, = \sum_{G \in \chn (\cA,\cB)}
    (-1)^{\lvert G \rvert -1} \, 
    \frac{\prod_{i=1}^{\lvert G \rvert - 1} \alpha (\cG_{i}, \cG_{i+1}) }
    {\prod_{j=1}^{\lvert G \rvert} \alpha (\cG_{j} , \cG_{j} ) } \ts ,
\]   
   with the usual convention that an empty product is\/ $1$.
\end{lemma}

\begin{proof}
  The claim can be established by induction on the maximal chain
  length in $\chn (\cA, \cB)$. Clearly, the formula gives $\alpha^{-1}
  (\cA,\cA) = 1/\alpha(\cA,\cA)$, where $\chn (\cA,\cA)$ consists of
  $G=\{\cA\}$ only. If $\cA \prec \cB$ with the property that the
  maximal chain length between them is $2$, which means that
  $G=\{\cA,\cB\}$ is the only chain from $\cA$ to $\cB$, the formula
  gives
\[
    \alpha^{-1} (\cA,\cB) \, = \, - \, 
    \frac{\alpha (\cA,\cB) } {\alpha (\cA, \cA)\, \alpha (\cB,\cB) }
\]   
which clearly agrees with Eq.~\eqref{eq:rec-inverse} for this case.

Assume now that the formula holds for any pair with maximal chain
length $k\leqslant n$, and consider $\cA \prec \cB$ with maximal chain
length $n+1$. Then, Eq.~\eqref{eq:rec-inverse} gives the relation
\[
   \alpha^{-1} (\cA, \cB) \, = \, - \sum_{\cA \prec \udo{\cC} \preccurlyeq \cB}
   \frac{\alpha (\cA, \cC)}{\alpha(\cA, \cA)} \,
    \alpha^{-1} (\cC, \cB) \ts .
\]
Now, by construction, the maximal chain length between any $\cC$ and
$\cB$ in the sum is bounded by $n$, so that we can use the induction
hypothesis as follows,
\[
\begin{split}
   \alpha^{-1} (\cA, \cB) \, & = 
   \, - \!\sum_{\cA \prec \udo{\cC} \preccurlyeq \cB}
   \frac{\alpha (\cA, \cC)}{\alpha(\cA, \cA)} 
   \sum_{G \in \chn(\cC,\cB)} (-1)^{\lvert G \rvert - 1} \,
    \frac{\prod_{i=1}^{\lvert G \rvert - 1} \alpha (\cG_{i}, \cG_{i+1}) }
    {\prod_{j=1}^{\lvert G \rvert} \alpha (\cG_{j} , \cG_{j} ) } \\[1mm]
   & = \sum_{\{ \cA \} \dot{\cup} \udo{G} \in \chn (\cA,\cB)} (-1)^{\lvert G \rvert}
    \, \frac{\alpha (\cA,\cG_{1})}{\alpha (\cA,\cA)} \,
    \frac{\prod_{i=1}^{\lvert G \rvert - 1} \alpha (\cG_{i}, \cG_{i+1}) }
    {\prod_{j=1}^{\lvert G \rvert} \alpha (\cG_{j} , \cG_{j} ) } \\[1mm]
    & = \sum_{H\in \chn (\cA, \cB)} (-1)^{\lvert H \rvert - 1} \,
     \frac{\prod_{i=1}^{\lvert H \rvert - 1} \alpha (\cH_{i}, \cH_{i+1}) }
    {\prod_{j=1}^{\lvert H \rvert} \alpha (\cH_{j} , \cH_{j} ) }  \ts ,
\end{split}
\]
where the calculation really is a consequence of how each chain from
$\cA$ to $\cB$ can uniquely be decomposed into one step from $\cA$ to
$\cC$, with $\cA \prec \cC \preccurlyeq \cB$, and a chain from $\cC$
to $\cB$.
\end{proof}
\medskip

\section*{Acknowledgements}
It is our pleasure to thank E.~Baake and M.~Salamat for valuable
discussions.  ES is grateful to the FSPM$^2$ of Bielefeld
University for financial support during her stay in Bielefeld.  This
work was also supported by the German Research Foundation (DFG),
within the SPP 1590. 

\end{document}